%% file: mcg.tex
\documentclass[11pt]{amsart}
\usepackage{amssymb}
\usepackage{amsmath} 
\usepackage{amsthm} 
\usepackage{epsfig}
\usepackage{graphicx}
\usepackage{color}
\usepackage{transparent}
\usepackage{soul}
\usepackage{subfig}
\usepackage{wrapfig}
\usepackage[colorinlistoftodos]{todonotes}
 
\numberwithin{equation}{section}

\newcommand{\tto}{\dashrightarrow}
\newcommand{\ppc}{\mathbf{P}^2(\mathbb{C})}
\newcommand{\ppr}{\mathbf{P}^2(\mathbb{R})}

\newtheorem*{thmm}{Main Theorem}

\newtheorem{thm}{Theorem}[section]
\newtheorem{lem}[thm]{Lemma}

\newtheorem{rem}[thm]{Remark}

\newtheorem{prop}[thm]{Proposition}

\newcommand{\Pic}{\text{Pic}}

\newcommand{\R}{\mathbb{R}}
\newcommand{\Z}{\mathbb{Z}}
\newcommand{\rptwo}{\mathbf{P}^2(\mathbb{R})}

\title{Mapping Classes of Real Rational Surface Automorphisms} 

\author{Kyounghee Kim}
\address{Department of Mathematics\\
         Florida State University\\
         Tallahassee, FL 32308}
\email{kkim6@fsu.edu}


\begin{document}
\maketitle

\begin{abstract}
Let $\{F_n, n\ge 8\}$ be a family of diffeomorphisms on real rational surfaces that are birationally equivalent to birational maps on $\ppr$. 
In this article, we investigate the mapping classes of the diffeomorphisms $F_n, n\ge 8$. These diffeomorphisms are reducible with unique invariant irreducible curves, and we determine the mapping classes of their restrictions, $\hat F_n, n \ge 8$, on the cut surfaces, showing that they induce pseudo-Anosov maps on a puctured oriented surface and do not arise from Penner's construction. For $n=8$, Lehmer's number is realized as the stretch factor of $\hat F_8$, a pseudo-Anosov map on a punctured genus $5$ orientable surface. The diffeomorphism $\hat F_8$ is a new geometric realization of Lehmer's number.
\end{abstract}

\section{Introduction}\label{S:intro}
\input{intro}

\section{From Real Rational Surface Automorphisms}\label{S:rat}
\input{raut}

\section{From Dehn twists}\label{S:twists}
\input{dehn}

\section{Proof of Main Theorem}\label{S:thmA}
\input{thma}

\section{Pseudo-Anosov maps on $\hat X_n$}\label{S:psa}
\input{psa}

\bibliographystyle{plain}
\bibliographystyle{unsrt}
\bibliography{biblio}
\end{document}

%% file: intro.tex
If an automorphism $F: X(\mathbb{C}) \to X(\mathbb{C})$ on a rational surface preserves a set $X(\mathbb{R})$ of real points, the restriction $F|_{X(\mathbb{R})}$ is a diffeomorphism on a real rational (non-orientable) surface $X(\mathbb{R})$. According to Nielsen-Thurston's classification \cite{Thurston:1988, Bleiler-Casson:1988}, the mapping class group of a finite type surface consists of one of the three types: periodic, reducible, and pseudo-Anosov. 

\vspace{1ex}
In this article, we consider a family $\{ F_n: X_n(\mathbb{R}) \to X_n(\mathbb{R}), n\ge 8\}$ of real rational surface automorphisms with unique invariant irreducible curves $C_n$. Due to the existence of a unique invariant curve, the mapping class of each member is reducible. We focus on the restriction,  $\hat F_n= F_n|_{X_n(\mathbb{R})\setminus C_n} $ on a cut surface and investigate its mapping class. For each $n\ge 8$, we identify the pseudo-Anosov map given by a product of positive Dehn twists, which is isotopic to $\hat F_n$. We demonstrate this result by comparing their induced actions on the fundamental group $\pi_1(X(\mathbb{R})\setminus C_n)$.

\vspace{1ex}

A birational map $f: \ppc \dasharrow \ppc$ is a rational map such that there exists a birational map $g:\ppc \dasharrow \ppc$ with $f\circ g = Id$ on a Zariski dense subset of $\ppc$. We refer to $g= f^{-1}$ as the inverse of $f$. Homogeneous polynomials determine a birational map:
\[ f [t:x:y] = [f_t:f_x:f_z] \] where coordinate functions $f_t, f_x, f_y$ can be written as homogeneous polynomials of the same degree $d$ with no non-constant common factor. The (algebraic) degree of $f$ is defined by the common degree $d$. One important quantity for a birational map is \textit{ a dynamical degree} $\lambda(f)$: 
\[ \lambda(f) = \lim_{n \to \infty} (\text{degree } f^n)^{1/n}. \]

 The dynamical degree can determine whether a given birational map is birationally equivalent to an automorphism.

\begin{thm}\cite{Diller-Favre:2001, blancdynamical}
Suppose $f$ is a birational map with the dynamical degree $\lambda(f)$.
\begin{itemize}
\item If $\lambda(f)$ is a Salem number, then $f$ is birationally equivalent to an automorphism.
\item If $\lambda(f)$ is a non-quadratic Pisot number, then $f$ is not birationally equivalent to an automorphism. 
\end{itemize}
\end{thm}
A Salem number $\lambda>1$ is an algebraic integer whose Galois conjugates, except for $\lambda$ and $\lambda^{-1}$, have modulus $1$. A Pisot number $\lambda>1$ is an algebraic integer whose Galois conjugates, except for $\lambda$, have modulus strictly less than $1$. In the case where a birational map $f: \ppc \dasharrow \ppc$ is birationally equivalent to an automorphism $F: X(\mathbb{C}) \to X(\mathbb{C})$ on a rational surface $X(\mathbb{C})$, the dynamical degree $\lambda(f)$ is given by the spectral radius of the induced action $F_*$ on $H^{1,1}(X(\mathbb{C}))$. In this case, the topological entropy $h_{top} (F)$ of the automorphism is equal to $\log \lambda(f)$. 

\vspace{1ex}
Let  $f_n :\ppr \tto \ppr , n\ge 8$ be a quadratic birational maps on $\ppr$ defined by 
\begin{equation}\label{E:int} f_n [t:x:y] \ =\   \left[ t x: \frac{\lambda_n^3}{1+\lambda_n} (y-t) x + tx: (\lambda_n^3+\lambda_n^2) (y-x) t + tx\right],   \end{equation}
where $\lambda_n$ is the largest real root of $ \chi_n(t) = t^{n} (t^3-t-1) + t^3+t^2-1$. This polynomial $\chi_n(t), n\ge 8$ is a product of Salem polynomial and cyclotomic factors \cite{Bedford-Kim:2006, gross2009cyclotomic}. Thus, for each $n\ge 8$, $\lambda_n$ is a Salem number. In fact, $\lambda_8$ is Lehmer's number, the known smallest Salem number. 

The birational map $f_n, n\ge 8$, has been considered by many authors \cite{Bedford-Kim:2006, McMullen:2007, Diller-Kim, Kim-Klassen}. For each $n$, the birational map $f_n$ has a unique invariant cubic curve $C_n$ with a cuspidal singularity. In \cite{McMullen:2007, Diller:2011}, a systematic construction was provided using the unique invariant curve, showing that the coordinate functions of $f_n$ are polynomials in $\mathbb{Z}(\lambda_n)[t,x,y]$.

\vspace{1ex}
A birational map $f_n$ is birationally equivalent to an automorphism $F_n$ on a rational surface $X_n(\mathbb{R})$ obtained by blowing up $\ppr$ along $n+2$ points located on the set $C_{n, \text{reg}}$, the set of non-singular points on the unique invariant curve $C_n$. By computing the induced action $F_{n*}$ on $H_1( X_n(\mathbb{R}))$, it has been shown \cite{Diller-Kim} that the spectral radius of $F_{n*}|H_1( X_n(\mathbb{R}))$ is indeed $\lambda_n>1$, which is also the growth rate of the induced action on the fundamental group \cite{Kim-Klassen}. Additionally, it has been shown \cite{Diller-Kim} that there exists a non-vanishing meromorphic two form $\omega$ such that $F_{n*} \omega = \lambda_n \omega$ with $\lambda_n>1$.

\vspace{1ex}
Although rational surface automorphisms on $\ppc$ have been extensively studied by many authors \cite{dolgachev2008reflection,Blanc:2008,McMullen:2007, Bedford-Kim:2009,Bedford-Kim:rotation}, many of these methods are not applicable to rational surface automorphisms on $\ppr$. One of the primary objectives of this article is to introduce a new approach to studying the dynamics of real rational surface automorphisms. For instance, by treating them as pseudo-Anosov maps, we can apply a train-track algorithm \cite{bestvina:1995} and examine the invariant foliations as suggested in \cite{Boyland:2024}.

\vspace{1ex}
The invariant cubic curve $C_n$ is a simple closed curve. By taking the class of $C_n$ as a generator of the fundamental group, we can observe that the restriction to the cut-surface $X_n(\mathbb{R}) \setminus C_n$ also has $\lambda_n$ as the growth rate of the induced action. 

Pseudo-Anosov mapping classes on non-orientable surfaces have been investigated in many articles \cite{Strenner:2018,Liechti-Strenner:2021,Khan-Partin-Winarski}. Liechti and Strenner \cite{Liechti-Strenner:2021} proved an interesting result regarding the stretch factor of a pseudo-Anosov map on a non-oriented surface. 
\begin{thm}\cite[Theorem~1.10]{Liechti-Strenner:2021}
If $f$ is a pseudo-Anosov map on a nonorientable surface or an orientation-reversing pseudo-Anosov map on an orientable surface, then the stretch factor of $f$ does not have Galois conjugate on the unit circle. 
\end{thm}

According to the result of Nicholls, Scherich, and Shneidman \cite{Nicholls-Scherich-Shneidman:2023}, the cut surface $X_n(\mathbb{R}) \setminus C_n$ turns out to be an orientable surface. Since all the centers of blowups lie on the cubic $C_n \subset \ppr$, we see that $C_n$ passes through every cross-cap exactly once. A cross-cap is a M\"{o}bius band glued to a surface. Therefore, a genus $g$ non-orientable surface, a connected sum of $g$ copies of $\ppr$, can be viewed as a sphere with $g$ cross-caps. 

\begin{thm}\cite[Proposition~3.5]{Nicholls-Scherich-Shneidman:2023}
Let $X$ be a genus $g$ non-orientable surface, and let $c$ be a curve in $X$ that passes through every cross-cap exactly once. The result of cutting along $c$ yields an orientable surface. 
\end{thm}

Let $\hat X_n = X_n(\mathbb{R}) \setminus C_n$ be a cut surface, and $\hat F_n= F_n|_{X_n(\mathbb{R}) \setminus C_n}$ denote the restriction on the cut surface. Depending on the number of blowups, a cut surface $\hat X_n$ is homeomorphic to either an orientable surface with one puncture or one with two punctures. 
Using a product of positive Dehn twists along a set of two-sided simple closed curves on the cut surface, we got the following result regarding the mapping classes of real rational surface automorphisms:

\begin{thmm}
Let $F_n : X_n(\mathbb{R}) \to X_n(\mathbb{R})$ be the diffeomorphism associated with the birational map \(f_n\) on $\ppr$ defined in~\eqref{E:int}.  For every \(n \ge 8\) the following hold:
\begin{itemize}
  \item $F_n$ is reducible and preserves a simple closed curve \(C_n\).
  \item The cut surface
        \[
            \widehat{X}_n := X_n(\mathbb{R}) \setminus C_n
        \]
        is orientable of genus \(\lfloor (n+2)/2 \rfloor\); it has one puncture when \(n\) is even and two punctures when \(n\) is odd.
  \item There exists a punctured orientable surface
        \(\widetilde{X}_n\) of the same genus such that the induced map
        \[
            \widehat{F}_n : \widehat{X}_n \longrightarrow \widehat{X}_n
        \]
        descends to a pseudo-Anosov map
        \[
            \widetilde{F}_n : \widetilde{X}_n \longrightarrow \widetilde{X}_n
        \]
        with stretch factor \(\lambda_n\).  Moreover, \(\widehat{F}_n\) is isotopic to the composition of \(n+2\) positive Dehn twists described in~\eqref{E:twist} of Section~\ref{S:twists}.
  \item The map \(\widetilde{F}_n\) is \emph{not} obtained from Penner's construction.
\end{itemize}
\end{thmm}

\begin{rem}
Each puncture of \(\widetilde{X}_n\) is either fixed or periodic under \(\widetilde{F}_n\).  Treating the punctures as marked points, we may regard \(\widetilde{F}_n\) as a pseudo-Anosov homeomorphism of an orientable surface of genus \(g = \lfloor (n+2)/2 \rfloor\).
\end{rem}

%

\begin{rem} The homeomorphism $\hat F_8$ on a cut surface $\hat X_8$ is pseudo-Anosov with the stretch factor $\lambda_8\approx 1.17628$, which is the Lehmer's number $L_\lambda$, the known smallest Salem number. The minimal polynomial $L(t)$ of the Lehmer's number is called the Lehmer's polynomial:
\[ L(t)= t^{10}+t^9-t^7-t^6-t^5-t^4-t^3+t+1.\]

  The Lehmer's number has fascinating applications across a diverse range of mathematical areas. For example, it appears in the Alexander polynomial for the $(-2,3,7)$ pretzel knot and represents the smallest growth rate of a hyperbolic polygonal reflection group \cite{Hironaka:2001}. It also holds the smallest leading eigenvalue among all essential elements in the Coxeter group $W(E_n)$ corresponding to the Coxeter diagram $E_n$ with $n\ge10$ \cite{McMullen:2002}. Additionally, it signifies the smallest entropy among all (compact, complex) surface automorphisms with positive entropy \cite{McMullen:2007}. 
    
  \vspace{1ex}
  Bedford and Kim \cite{Bedford-Kim:2006} constructed an automorphism on a rational surface with the entropy $\log L_\lambda$. McMullen \cite{McMullen:2016} achieved a similar goal on a $K3$ surface, and Diller and Kim  \cite{Diller-Kim} constructed an automorphism on a real rational surface with the same entropy $\log L_\lambda$.
  
 Our example, $\hat F_8$,  introduces a new realization of Lehmer's number as the stretch factor of a pseudo-Anosov map on a once-punctured orientable genus $5$ surface (equivalently, an orientable genus $5$ surface with one market point). There is a $18$-prong singularity at the marked point \cite{HironakaE:2010,Farb-Margalit:2012}, and we can recover the original map $F_8$ by blowing up the marked point. Note that this map is given by the explicit formula in (\ref{E:int}) with $n=8$ and $\lambda_8= L_\lambda$.
\end{rem}


Let \(W(E_n)\) denote the Weyl group of type \(E_n\). It acts by isometries of the real Lorentzian space \(\mathbb{R}^{1,n}\) endowed with the bilinear form of signature \((1,n)\) (see \cite{McMullen:2007}). 
For each \(n\), the diffeomorphism \(F_n\) from the main theorem can be seen as an automorphism of a rational surface \(X_n(\mathbb{C})\). 
Via the intersection pairing, \(\Pic(X_n(\mathbb{C}))_\mathbb{R}\) identifies with a Lorentzian space of type \(E_{n+2}\); under this identification, the pullback \(F_n^*:\Pic(X_n(\mathbb{C}))\!\to\!\Pic(X_n(\mathbb{C}))\) is conjugate to the action of a Coxeter element of \(W(E_{n+2})\) (i.e. the product of the simple reflections for the \(E_{n+2}\) diagram).

The spectral radius $\lambda(\omega)$ is the leading eigenvalue of $\omega$. Let $\Lambda(W)$ be the set of spectral radii of elements of the Weyl group $W(E_n), n\ge 10$ associated with the Coxeter diagram $E_n$.  For any $\lambda>1$ in $\Lambda(W)$, assuming the existence of a unique invariant cubic with a cuspidal singularity, one can construct an automorphism on a rational surface over $\mathbb{C}$ such that its dynamical degree is equal to $\lambda$ and the automorphism preserves the surface over real projective plane \cite{Uehara:2010,McMullen:2007,Kim:2024}. For these maps, we have only partial results about the growth rate of the induced action of the real restriction on the fundamental group \cite{Kim-Klassen, Kim-Park}. For certain quadratic rational surface automorphisms, Kim and Klassen \cite{Kim-Klassen} and Kim and Park \cite{Kim-Park} showed that the growth rates of the induced action on the fundamental group are strictly greater than $1$. Since all the centers of the blowups are again on the unique invariant cubic, we can also obtain diffeomorphisms on orientable cut surfaces. We leave the investigation of these maps for future research.

\vspace{1ex}
This article is organized as follows. In Section \ref{S:rat}, we discuss the family of quadratic rational surface automorphisms and diffeomorphisms on cut surfaces. Using \textit{reading curves}, we computed the induced actions on the fundamental group of those diffeomorphisms. In Section \ref{S:twists}, we determine the sets of simple closed curves on cut surfaces whose incident graphs are given by $E_n$ diagram. Then, we define the products of the positive Dehn twists along these sets and compute the action on the fundamental group. The proof of the Main Theorem is given in Section \ref{S:thmA}. The discussion regarding more diffeomorphisms associated with quadratic rational surface automorphisms is in Section \ref{S:psa}.

\subsection*{Acknowledgement} We would like to thank Samuel A. Ballas, Eko Hironaka, and Curtis McMullen for their valuable discussions and keen interest in this project.

%% file: raut.tex
Consider  $\{f_\alpha :\ppc \tto \ppc, \alpha \in \mathbb{C} \setminus \{0,-1\} \}$ to be a family of birational maps on $\ppc$ defined by 
\begin{equation}\label{E:feq} f_\alpha[t:x:y] \ =\   \left[ t x: \frac{\alpha^3}{1+\alpha} (y-t) x + tx: (\alpha^3+\alpha^2) (y-x) t + tx\right].  \end{equation}
Each $f_\alpha$ can be written as a composition of a linear map $L_\alpha$ and the involution $J$: 
\[ f_\alpha = L_\alpha \circ J\]  where $J : [t:x:y] \mapsto [xy:ty:tx]$ and 
 \[ L_\alpha = \begin{bmatrix} 0&0&1 \\ \alpha^3/(1+\alpha) & 0 & 1-\alpha^3/(1+\alpha) \\ 0 & \alpha^3+\alpha^2& 1- \alpha^3-\alpha^2 \end{bmatrix}.\]
Thus the inverse $f^{-1}_\alpha$ is given by $J \circ L_\alpha^{-1}$.
In this family, all birational maps $f_\alpha$ share the same set of exceptional lines and the same indeterminacy locus. The set of exceptional lines consists of three lines, \[ \mathcal{E}(f_\alpha) \ = \{ L_t = \{t=0\}, L_x= \{x=0\}, L_y = \{ y=0\} \}.\] 
The intersections of the coordinate axis give three points of indeterminacy: \[ \mathcal{I}(f_\alpha) \ = \  \{ [1:0:0],[0:1:0],[0:0:1] \}.\] 
Under $f_\alpha$, each exceptional line is mapped to a point 
\[ \begin{aligned} f_\alpha \ :\ & L_t \mapsto [0:1:0],\\ &L_x \mapsto [0:0:1],\\ &L_y\mapsto [1:1-\alpha^3/(1+\alpha): 1-(\alpha^3+\alpha^2)].\\ \end{aligned}\]
Using $f^{-1}_\alpha$, we see that there are three exceptional lines for $f^{-1}_\alpha$
\[ \mathcal{E}(f^{-1} _\alpha) \ = \{ M_t =L_\alpha \{t=0\}, M_x= L_\alpha \{x=0\}, M_y =L_\alpha \{ y=0\} \}, \text{and}\]
 \[ f^{-1}_\alpha \ :\ M_t \mapsto [1:0:0],\quad M_x \mapsto [0:1:0],\quad M_y\mapsto [0:0:1].\]
 
\vspace{1ex}
For a generic $\alpha  \in \mathbb{C} \setminus \{0,-1\}$, $f_\alpha$ is not birationally equivalent to an automorphism and its dynamical degree is given by a Pisot number $\rho \approx 1.32472$, where $\rho$ is the plastic ratio, which is the largest real root of $t^3-t-1$. However, if $\alpha$ satisfies a certain polynomial equation, there is a rational surface on which the lift of $f_\alpha$ is an automorphism. Such $\alpha$ can be found using the following theorem. 
\begin{thm}\cite{Bedford-Kim:2006, Kim:2024}
Suppose $f: \ppc \dasharrow \ppc$ is a birational map. Let $\mathcal{E}(f) = \{ C_1, \dots, C_m\}$ be the set of irreducible exceptional curves (counted with multiplicity) of $f$ and $\mathcal{I}(f)$ be the set of points of indeterminacy. The birational map $f$ is birationally equivalent to an automorphism $F:X \to X$ on a rational surface if and only if there is a set of positive integers $\{ n_1, \dots, n_m\}$ such that for all $i =1, \dots, m$
\begin{itemize}
\item the image $f^{n_i}(C_i)$of $n_i$-fold composition is a point of indeterminacy in $\mathcal{I}(f)$,
\item for all $1 \le j \le n_i-1$, $f^{j}(C_i) \not\in \mathcal{I}(f)$.
\end{itemize}
Furthermore, the rational surface $X$ is a blowup of $\ppc$ along a set of (possibly infinitely near) points $\{ f^{j}(C_i),  j=1, \dots, n_i, i=1, \dots m\}$.
\end{thm}
For quadratic birational maps defined in (\ref{E:feq}), the conditions are given by polynomial equations in $\alpha$.
Here, we only list such polynomials and discuss some of the properties of corresponding $f_\alpha$'s. For more detailed information, please refer to the articles \cite{Bedford-Kim:2009, McMullen:2007, Diller:2011, Kim:2022}. 

\begin{thm} \cite{Bedford-Kim:2009, McMullen:2007, Diller:2011, Kim:2022}\label{T:rr}
Let $\{ \chi_n(t) = t^{n} (t^3-t-1) + t^3+t^2-1, n \ge 8\}$ be a discrete family of polynomials. Suppose $\alpha$ is a root of $\chi_n$ for some $n\ge 8$. Then we have:
\begin{enumerate}
\item $\chi_n(t)$ are Salem polynomials such that their largest real roots $\lambda_n$ approach to the plastic ratio from below:
\[ \lim_{n \to \infty} \lambda_n \ = \ \rho,\]
\item $\lambda_8 \approx 1.17628$ is Lehmer's number, the smallest known Salem number,
\item $f_\alpha$ has a unique invariant curve $C_\alpha$, which is  a cubic with a cuspidal singularity. 
\item $f_\alpha$ has exactly two fixed points: one at the cusp of $C_\alpha$ with two multipliers, $\alpha^{2}$ and $\alpha^{3}$, and one at a non-singular point on $C_\alpha$ with two multipliers, $\alpha^{-1}$ and $\alpha^{n-3}$,
\item the restriction map $f_\alpha|_{C_\alpha}$ is conjugate to a linear map defined by $t \mapsto (1/\alpha) t$,
\item $f_\alpha$ is birationally equivalent to a rational surface automorphism $F_\alpha: X_\alpha \to X_\alpha$ whose topological entropy is given by $\log \lambda_n$, and 
\item the rational surface $X_\alpha$ is obtained by blowing up the finite set $P_\alpha$ of $n+2$ points on the unique invariant cubic $C_\alpha$ of $f_\alpha$, where the point set $P_\alpha$ consists of orbits of exceptional lines:
\[ P_\alpha = \{ f_\alpha L_t = [0:1:0], f_\alpha L_x = [0:0:1], f^j_\alpha L_y, j=1, \dots, n-1, f^n_\alpha L_y=[1:0:0]\}.\]
\end{enumerate}
\end{thm}

\subsection{Diffeomorphisms on Real Rational surfaces}\label{SS:diffeo} Each polynomial $\chi_n$ has exactly two real roots, $\lambda_n$ and $1/\lambda_n$, outside the unit circle. If we set $\alpha = \lambda_n$ or $1/\lambda_n$, then $f_\alpha$ preserves the real projective plane, and its modification $F_\alpha$ preserves the surface over $\ppr$. Since $f_{1/\lambda_n}$ is conjugate to $f^{-1}_{\lambda_n}$, in this article, we will focus on $\alpha= \lambda_n$ and the corresponding mappings restricted to real surfaces. To simplify the notation, let us use $n$ as the subscript in the cases of $\alpha= \lambda_n$: $X_n \subset X_\alpha$ denotes the surface over $\ppr$, 
\[ f_n = f_{\lambda_n}|_{\ppr}, \quad F_n=  F_{\lambda_n}|_{X_n}, \quad etc.\]

The exact locations of center blowups on the invariant cubic curve $C_n$ and the relative position of exceptional lines $L_t,L_x, L_y$ and $M_t, M_x, M_y$ are essential for the rest of the paper. To have those, let us restate Diller's theorem in our setting. We work with the inverse map so that the cusp fixed point is repelling; to streamline signs we make a harmless reparameterization. All formulas are therefore written for \(f^{-1}\) with this parameter convention. The following modified version is due to the uniqueness up to linear conjugacy in Diller's original theorem.

\begin{thm}\cite{Diller:2011,Diller-Kim}
Let $\lambda_n$, $n\ge 8$ be the largest real root of a polynomial $\chi_n(t) = t^{n} (t^3-t-1) + t^3+t^2-1$. Let $f_n$ be the map defined in \eqref{E:feq} with $\alpha = \lambda_n$. Then there exists an affine change of parameter $\gamma (t)$ for the invariant cubic curve $C_n$ such that 
\begin{enumerate}
\item $\gamma(0)$ is the unique saddle type fixed point and $\gamma(\infty)$ is the repelling fixed point at the cusp.
\item The restriction map $f_n|_{C_n}: \gamma(t) \mapsto \gamma ( (1/\lambda_n) t)$
\item Centers of blowups are given by \[\gamma(a_i) := f_n^i L_y, i=1, \dots, n \ \ \gamma(b_1) : = f_n L_x,\ \ \gamma(c_1):= f_n L_t\] where 
\[ a_i = \lambda_n^{n-j} \frac{ 1+ \lambda_n+ \lambda_n^2}{\lambda_n^{2+n}-1},\ i=1, \dots, n, \quad b_1 = \frac{ 1+ \lambda_n^n+ \lambda_n^{1+n}}{\lambda_n^{2+n}-1}, \quad c_1 = \frac{ 1+ \lambda_n+ \lambda_n^{1+n}}{\lambda_n^{2+n}-1}\]
\end{enumerate}
\end{thm}


\paragraph{Notation.}
By abuse of notation, we use the same symbols \(a_i,b_1,c_1\) both for the marked points on \(C_n\) and for their parameter values.
We also write \(a_0,a_{n+1},b_0,b_2,c_0,c_2\) for the predecessor/successor points along the \(f_n|_{C_n}\)-orbit, so on parameters
\[
a_{j+1}=(1/\lambda_n) a_j\ (0\le j\le n),\qquad b_2=(1/\lambda_n )b_1,\qquad c_2=(1/\lambda_n) c_1,
\]
and hence \(a_0=\lambda_n a_1\), \(b_0=\lambda_n b_1\), \(c_0=\lambda_n c_1\).
Additionally, let us use $C_n$ for invariant cubics for both $f_n$ and $F_n$.

\begin{rem}
Since \(f_n\) lifts to an automorphism of a rational surface, the induced action on the Picard group preserves the intersection form and the canonical class. Together with the fact that \(f_n\) has exactly three exceptional lines in \(\mathbf{P}^2\), this implies that the centers of the blowups are in general position (no three collinear, no six lying on a conic, etc.). In particular, the invariant cubic \(C_n\) is defined over \(\mathbb{R}\); equivalently, its real locus \(C_n(\mathbb{R})=C_n\cap \mathbf{P}^2(\mathbb{R})\) is a real cubic.
\end{rem}

\begin{rem}
Work in homogeneous coordinates \([t:x:y]\) and set \(L_x=\{x=0\}\), \(L_y=\{y=0\}\), \(L_t=\{t=0\}\).
From the formula for \(f_n\) in \eqref{E:feq}, the line \(L_x\) passes through
\(b_1=[0:0:1]\) and \(a_n=[1:0:0]\).
Hence there is a third point \(p\in L_x\cap C_n\).
Since \(p\) is not an indeterminacy point, we have \(f_n(p)=\gamma(b_1)\in C_n\).
As \(C_n\) is invariant and \(f_n|_{C_n}\) acts by the parameter rule (so that \(b_0\) is the predecessor of \(b_1\)),
it follows that \(p=b_0\).
Thus
\[
\{a_n,\,b_1,\,b_0\}\subset L_x.
\]
By the same argument we obtain
\[
\{b_1,\,c_0,\,c_1\}\subset L_t,
\qquad
\{a_0,\,a_n,\,c_1\}\subset L_y.
\]

For the exceptional lines of \(f_n^{-1}\), denoted \(M_x,M_y,M_t\), the corresponding triples are
\[
\{a_1,\,c_1,\,c_2\}\subset M_x,\qquad
\{a_1,\,a_{n+1},\,b_1\}\subset M_t,\qquad
\{b_1,\,b_2,\,c_1\}\subset M_y,
\]
in accordance with the action of \(f_n|_{C_n}\) on the parameter along \(C_n\).
\end{rem}

\begin{rem}\label{R:basept}
Notice that $f_n$ is a diffeomorphism outside the three lines $L_t, L_x, L_y$. Each line $L_t, L_x, L_y$ maps to a intersection point of $M_t, M_x,M_y$.  In $\mathbf{P}^2(\mathbb{R})$, there are four connected components $T_i, i=1,..4$ enclosed by  $L_t, L_x, L_y$. Similarly there are four connected components $S_i, i=1, ..4$ enclosed by $M_t, M_x, M_y$. Since the intersection points of these three lines are points of indeterminacy which is the image of $M_t, M_x,M_y$ under $f_n^{-1}$, Each bounded region $T_i$ is mapped to one of bounded region $S_i$. 

On the invariant real cubic \(C_n(\R)\), the two fixed points \(\gamma(0)\) and \(\gamma(\infty)\) lie in the same component on both sides; relabel so that
\[
\gamma(0),\gamma(\infty)\in T_1 \quad\text{and}\quad f_n(T_1)=S_1 \ni \gamma(0),\gamma(\infty).
\]
Furthermore, since two multipliers at the cusp are $\lambda_n^2, \lambda_n^3$, $f_n$ preserves the orientation near $\gamma(\infty)$. 
It follows that there exist points \(p,q\in (T_1\cap S_1)\setminus C_n(\R)\) with \(f_n(p)=q\), and \(p\) and \(q\) can be joined by a path
\(\gamma_*\subset (T_1\cap S_1)\setminus C_n(\R)\).
\end{rem}

%

\vspace{1ex}
Using the induced action on the homology group, Diller and Kim \cite{Diller-Kim} showed that the topological entropy of $F_n$ is $\log(\lambda_n)$, and therefore, both the real and complex mappings have the same topological entropy. Later, Kim and Klassen \cite{Kim-Klassen} determined the induced action $F_{n*}: \pi_1(X_n,*) \to \pi_1(X_n,*)$ on the fundamental group. Let us summarize the relevant results.

\begin{thm}\cite{Diller-Kim, Kim-Klassen}\label{T:rat}
Real diffeomorphisms $F_n$ satisfy the following properties:
\begin{enumerate}
\item For each $n\ge 8$, the topological entropy of $F_n$ equals $\log \lambda_n$. 
\item The spectral radius of the induced action $F_{n*}$ on the homology group equals to $\lambda_n>1$.
\item $F_n$ has exactly two fixed points: one at the cusp of $C_n$, which is a repeller, and one at a non-singular point on $C_n$, which is a saddle type. 
\item The invariant cubic $C_n$ can be parametrized such that $C_n = \{ \gamma(t) : t \in \mathbb{R} \cup \{\infty\} \}$ where \[ \gamma(\infty) = \text{the cusp},\ \ \quad\gamma(0) = \text{the non-singular fixed point}, \] \[  \gamma(a_i) = f_n^i L_y,  \ \ i = 1, \dots, n, \qquad \gamma(b_1) = f_n L_x, \quad\text{and} \quad \gamma(c_1) = f_n L_t.\]
In this setting, all centers of blowups are ordered in the following way:
\[ \infty > a_1 > a_2 > b_1 >a_3 > a_4 > c_1 > a_5 > a_6 > \cdots >  a_{n-1} > a_n > 0.\]
\item There is a real meromorphic $2$ form $\omega$ on $X_n$ such that: 
\begin{itemize}
\item $\omega$ has a simple pole along the invariant cubic $C_n$ and no other poles or zeros,
\item $F_n^* \omega = \lambda_n \omega$ with $\lambda_n>1$.
\end{itemize} 
\end{enumerate}
\end{thm}

\subsection{Diffeomorphisms on Orientable Surfaces}\label{SS:cut}
Each surface $X_n$ is obtained by blowing up along a finite set of points on $\ppr$. This can be visualized as gluing the boundary of M\"{o}bius band to each circular boundary around the centers of blowups. A M\"{o}bius band glued to a surface is called a \textit{cross-cap}. it is useful to think about the cross-caps not
only as the boundary of a M\"{o}bius band glued to the inside dotted circle in Figure \ref{F:horizontal}, but also by quotienting the dotted circle by the antipodal map. The most common notation for a cross-cap is $\bigotimes$. In this article, we use a dotted closed curve for a cross-cap for easier labeling. 
Each surface $X_n$ is homeomorphic to a connected sum of $n+3$ copies of $\ppr$, a compact closed non-orientable surface $\mathcal{N}_{n+3}$ of genus $n+3$, and its Euler characteristic is $-1-n$:
\[ X_n \cong \mathcal{N}_{n+3}, \qquad  \chi(X_n) = -1-n.\]
Since all the centers of blowups lie on the invariant cubic $C_n \subset \ppr$, let us illustrate $X_n$ as in Figure \ref{F:horizontal}. The horizontal line in the middle represents the invariant cubic $C_n$. And each center of blowup is labeled in Figure \ref{F:horizontal} as follows: 
\[ a_i \ =\ f_n^i L_y, \ \ i = 1, \dots, n, \quad b_1 = f_n L_x, \quad c_1= f_n L_t. \]
Each dotted curve corresponds to the exceptional curve over the center of blowup and thus the boundary of M\"{o}bius band. Since the construction starts on $\ppr$, the outer oval curve is also dotted. Each closed dotted curve corresponds to one cross-cap, and we see that the invariant cubic $C_n$ crosses $n+3$ cross-caps exactly once. 

\begin{figure}
\includegraphics[width=4.5in]{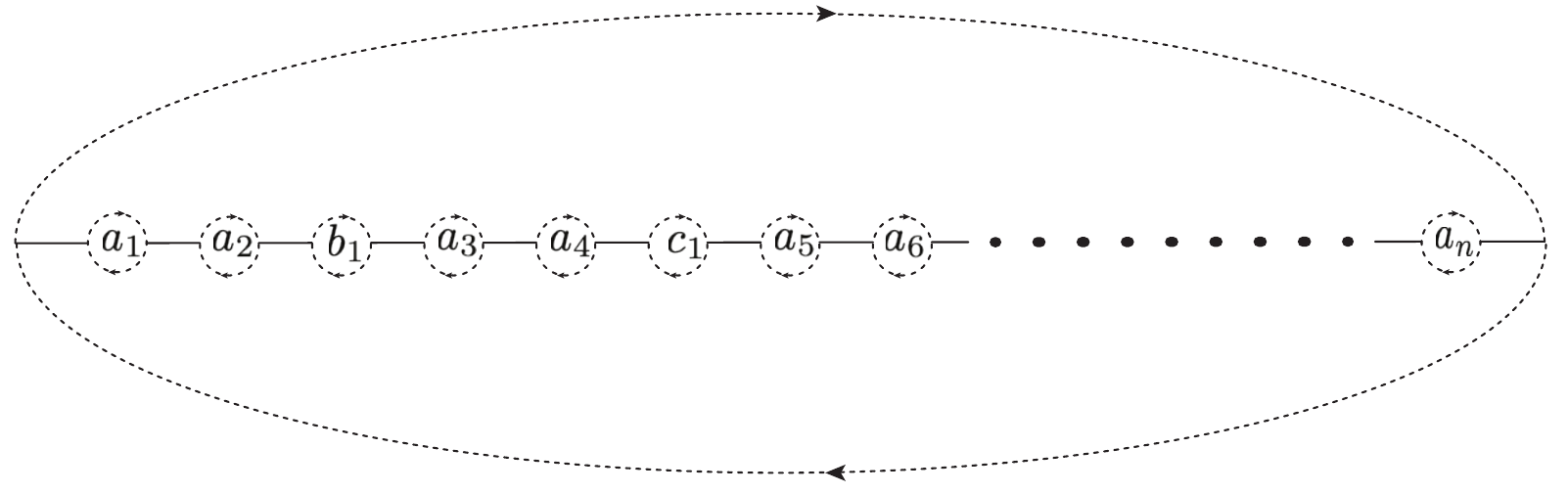}
\caption{$X_n$ with $C_n$ as a horizontal line}
\label{F:horizontal}
\end{figure}

\begin{lem} If $n$ is even, then the invariant cubic $C_n$ is a non-separating $1$-sided curve. If $n$ is odd, the invariant cubic $C_n$ is a non-separating $2$-sided curve. 
\end{lem}

\begin{proof}
It is clear that $C_n$ is non-separating from Figure \ref{F:horizontal}. If $n$ is even, $C_n$ passes through odd number $n+3$ cross-caps and thus $C_n$ is $1$-sided. Similarly, if $n$ is odd, $C_n$ is $2$-sided. 
\end{proof}

The cut surface of $X_n$ along the invariant cubic $C_n$ has either one boundary component if $n$ is even or two boundary components if $n$ is odd. Let us recall the result of Nicholls, Scherich, and Shneidman \cite{Nicholls-Scherich-Shneidman:2023}, which shows that the cut surface along $c$ is orientable by showing that there is no $1$-sided curve disjoint from $c$. 

\begin{prop}\cite[Proposition~3.5]{Nicholls-Scherich-Shneidman:2023}\label{P:cut}
Let $X$ be a genus $g$ non-orientable surface, and let $c$ be a curve in $X$ that passes through every cross-cap exactly once. The result of cutting along $c$ yields an orientable surface. 
\end{prop}

Let $\hat X_n = X_n \setminus \{ C_n\}$ be the complement of $C_n$. 
Since the invariant cubic $C_n$ passes through every cross-cap exactly once, the resulting surface $\hat X_n$ is homeomoprhic to a once punctured genus $(n+2)/2$ orientable surface, $M_{(n+2)/2,1}$, if $n$ is even, and $\hat X_n$ is homeomorphic to a twice punctured genus $(n+1)/2$ orientable surface, $M_{(n+1)/2,2}$, if $n$ is odd.
\begin{equation} \hat X_n \cong  M_{(n+2)/2,1} \ \ \text{if } n \text{ is even,} \qquad \text{and}  \ \ \hat X_n \cong  M_{(n+1)/2,2} \ \ \text{if } n \text{ is odd.}\end{equation}
Since the diffeomorphism $F_n$ preserves $C_n$, we have a homeomorphism $\hat F_n$ on $\hat X_n$ for $n\ge 8$.

\subsection{Reading curves} Kim and Klassen \cite{Kim-Klassen} showed that the induced action $F_n*$ on the fundamental group $\pi_1(X_n)$ is completely determined by the orbit of exceptional lines. In their article, the induced action was determined using a set of \textit{reading curves}, which are dual to the set of generators. Let us illustrate the idea of reading curves with a torus $\mathbb{T}$. 

\begin{figure}[h]
\includegraphics[width=3in]{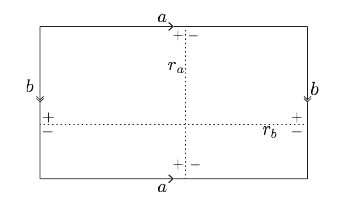}
\caption{The reading curves $r_a, r_b$ for a torus $\mathbb{T}$ with $\pi_1(\mathbb{T}) = \langle a, b| a b a^{-1} b^{-1} \rangle$}
\label{F:torus}
\end{figure}

\begin{figure}[h]
\includegraphics[width=2.86in]{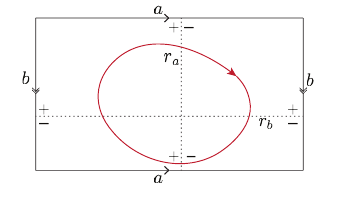}\ \ \includegraphics[width=3in]{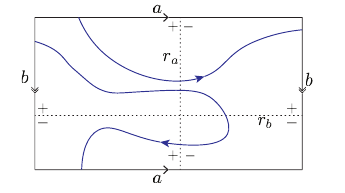}
\caption{The class of the red curve on the lhs  is given by $ab a^{-1}b^{-1} = 1$, and the class of the blue curve on the rhs is given by $a^2b a^{-1}= ab$}
\label{F:toruswC}
\end{figure}

The fundamental group of a torus has two generators $a$ and $b$ with a relation $a b a^{-1} b^{-1} = 1$. In the fundamental polygon, we can put the reading curve as shown in Figure \ref{F:torus}. Two reading curves $r_a, r_b$ are dotted lines inside the square. The $\pm$ signs indicate the positive or negative power of the corresponding generator. If a given curve $C$ crosses reading curve $r_a$ from the positive side to the negative side, then the $\pi_1$ class of $C$ will get $a$, and if it crosses $r_b$ from the negative side to the positive side, then $[C]$ will pick up $b^{-1}$. As long as a curve does not cross the unique intersection of all reading curves, one can easily determine the $\pi_1$ class of any given curve. Notice that one can always choose a set of reading curves such that the curve of interest does not pass through the intersection of reading curves.

We provide two simple examples in Figure \ref{F:toruswC}. The class of the left-hand side red curve is given by $a b a^{-1} b^{-1} = 1$, and the class of the right-hand side curve is given by $a^2b a^{-1}= ab$. See \cite{Coh, Kim-Klassen} for more details regarding reading curves. 

%
%

\medskip

A more illustrative example comes from a non-orientable surface. Let
\[X=\rptwo\#\rptwo\#\rptwo\#\rptwo\] and let \(c\) be a simple closed curve that meets each crosscap exactly once. Cutting along \(c\) yields the surface \(\hat X=X\setminus c\), and in this model
\(\pi_1(\hat X)\cong\Z*\Z*\Z\).
Figure~\ref{F:cutX} depicts \(\hat X\) as a hexagon with opposite sides identified (the dotted arcs near the corners indicate the cut \(c\); the opposite side identifications correspond to a M\"{o}bius-band pairing).

For \(\pi_1(\hat X,*)\) it is convenient to take generators to be the three loops that pass through exactly one crosscap; call them \(a,b,c\). We place the basepoint at their common intersection (see Figure~\ref{F:cutX}). The \emph{reading curves} are chosen dual to these generators: we use the sides of the hexagon as reading arcs. Fix the following sign convention. We orient the generator labeled \(a\) ``upward"; then an oriented curve that \emph{enters} the top side contributes \(a\), while entering the bottom side contributes \(a^{-1}\). The generators labeled \(b\) and \(c\) are treated similarly; the \(\pm\) signs printed next to those sides indicate whether the contribution is \(b^{\pm1}\) or \(c^{\pm1}\) according to the entering direction.

In Figure~\ref{F:cutXreading} we show two sample readings: the red curve on the left has class \(ab\), and the curve on the right has class \(c^{-1}a^{-1}\).

\begin{figure}
\includegraphics[width=2in]{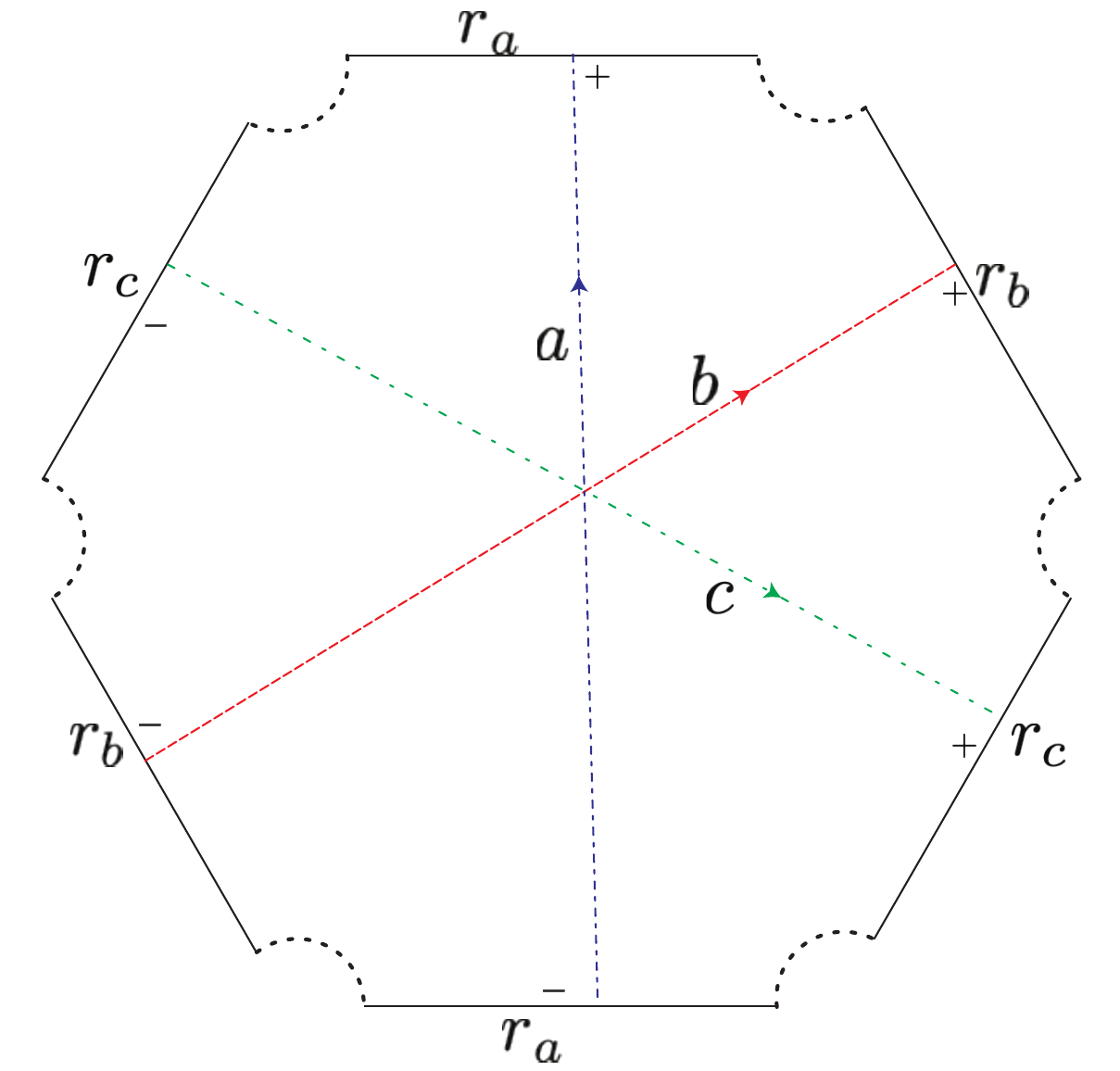}
\caption{The reading curves $r_a, r_b, r_c$ for a cut surface $\hat X$ with $\pi_1(\hat X)  = \langle a,b,c \rangle$}
\label{F:cutX}
\end{figure}

\begin{figure}
\includegraphics[width=2in]{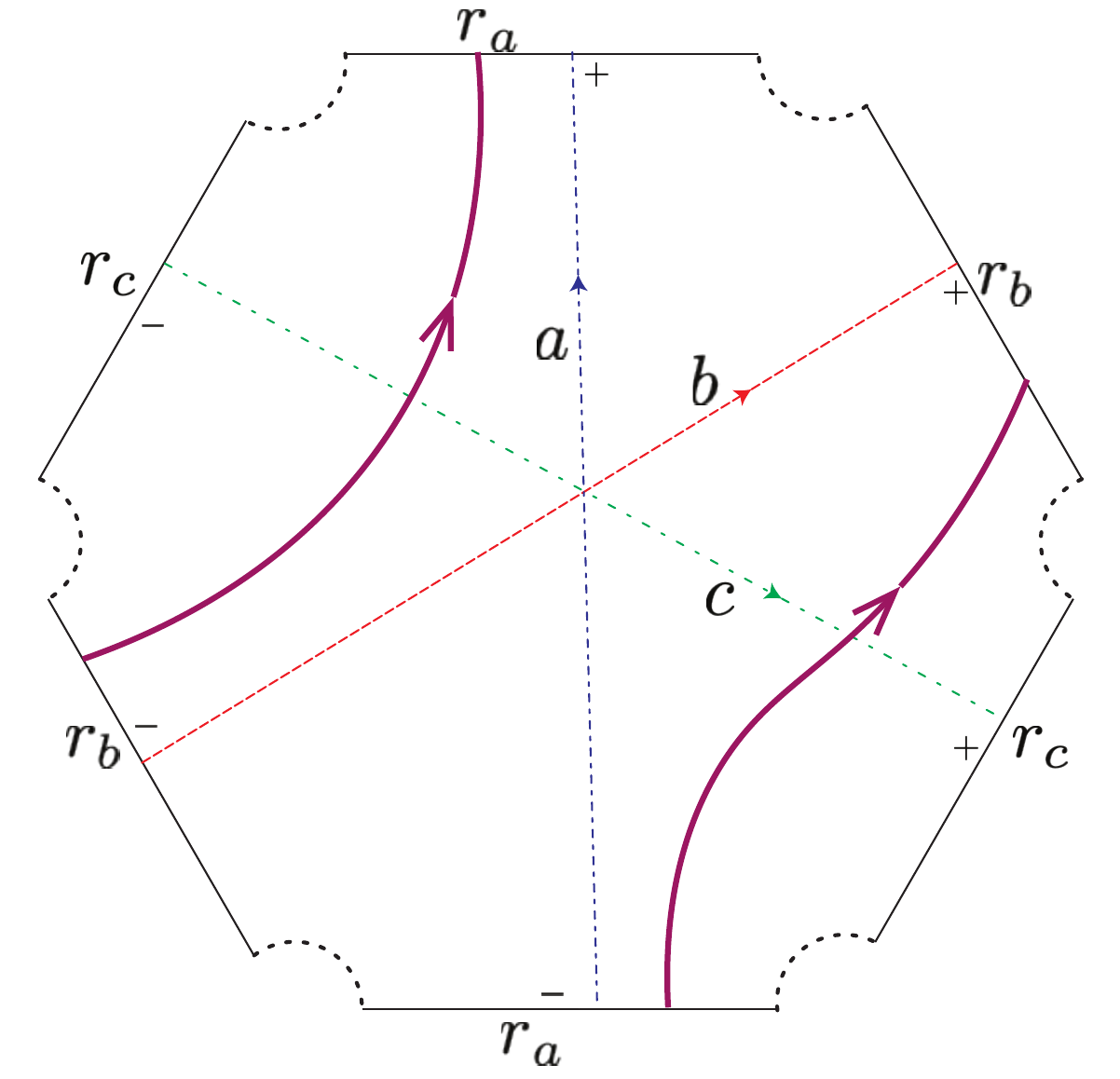}\ \quad \quad \includegraphics[width=1.93in]{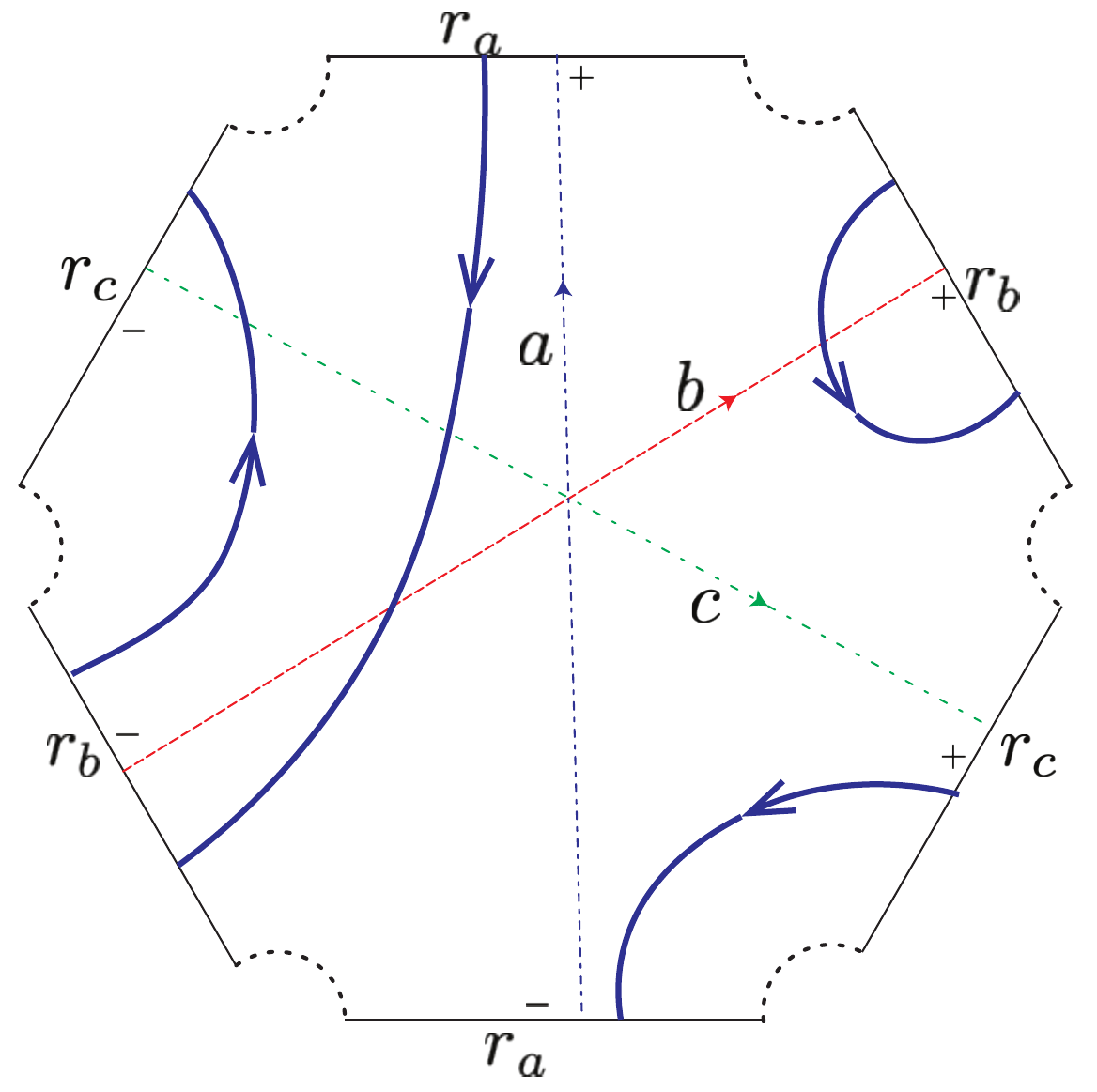}
\caption{The class of the red curve on the lhs  is given by $ab$, and the class of the blue curve on the rhs is given by $c^{-1} a^{-1} b^{-1} b= c^{-1}a^{-1}$}
\label{F:cutXreading}
\end{figure}

\subsection{The induced action on $\pi_1(\hat X_n)$}
In this article, the surface of interest is the cut-surface $\hat X_n$. By Proposition \ref{P:cut}, we see that $\hat X_n$ is homeomorphic to a $g=\lfloor (n+2)/2 \rfloor$ hole torus with one (if $n$ is even) or two (if $n$ is odd) punctures. Thus, the fundamental group $\pi_1(\hat X_n)$ is isomorphic to a free group $\mathbb{Z}*\mathbb{Z}* \cdots *\mathbb{Z}$, repeated $n+2$-times. 
\[ \pi_1(\hat X_n) = \mathbb{Z}*\mathbb{Z}* \cdots *\mathbb{Z}, \ \ n+2 \text{ times}.\]
We illustrate our choice of generators and reading curves in Figure \ref{F:gen}. In this figure, the middle horizontal line passing through all cross-caps will be removed, and the outside circle will be identified with the antipodal points. For our purpose, 
it is convenient to use the vertical curves passing through each exceptional curve as the set of generators and the middle horizontal curves inside each exceptional curve as reading curves, as shown in Figure \ref{F:gen}.  In other words, if a curve passes through each labeled cross-cap from top to bottom, it will get the positive power of the corresponding generator, and if a curve passes through the cross-cap from bottom to top, the corresponding generator has the negative power.  
Since $\hat F_n$ does not have a fixed point, we need to conjugate a path between the base point and its image to determine the image of any oriented curve under $\hat F_n$. Since the relative locations of exceptional lines $L_t, L_x, L_y$, $M_t,M_x,M_y$ and the removed invariant cubic are importnat to us, we will use the point $p$ discussed in Remark \ref{R:basept} as the base point and use the path $\gamma_*$ to move $q:=\hat F_n (p)$ back to the base point. Since $\gamma_*$ is disjoint from all exceptional curves, we can determine the image of simple closed curve using reading curves as in \cite{Kim-Klassen}.

\begin{figure}
\includegraphics[width=3in]{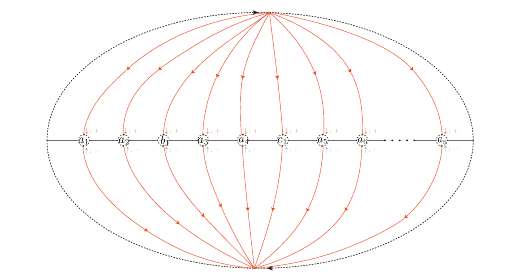}
\caption{The set of generators and the reading curves on $\hat X_n$}
\label{F:gen}
\end{figure}

We will determine the induced action on the fundamental group by calculating the class of the image of each generator under $\hat F_n$. Since $F_n$ preserves the invariant cubic, to determine the image of a curve $ \subset \hat X_n$ under $\hat F_n$, we may use the image under $F_n$ on $X_n$. It is worth pointing out that the mapping $F_n$ is determined by the three exceptional lines and the restriction map on the invariant cubic. (See \cite{McMullen:2007, Diller:2011}.) This makes the computation of the induced action of $\hat F_n$ straightforward.

\begin{figure}
\includegraphics[width=6in]{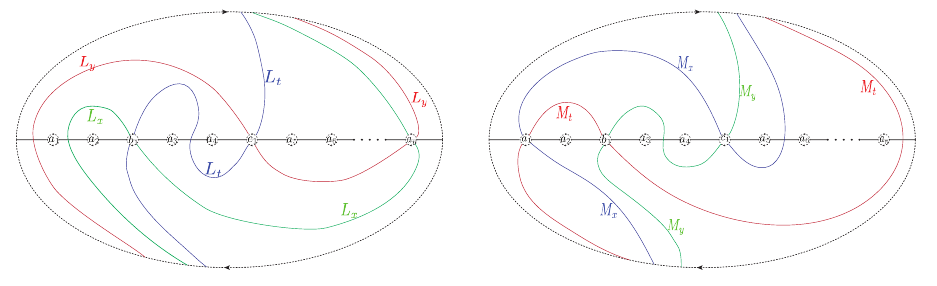}
\caption{LHS: $X_n$ with exceptional lines for $f_n$, RHS:  $X_n$ with exceptional lines for $f^{-1}_n$ }
\label{F:exceptional}
\end{figure}

 Let us highlight a couple of useful properties of $F_n$. To simplify the notations, we will use the same notation for the curves in $\ppr$ and their strict transforms on $X_n$. We will also use $a_i, i=1, \dots, n, b_1, c_1$ to denote exceptional curves, the corresponding centers of blowups, and the corresponding generators.  Notice that the strict transforms of three lines $L_t, L_x, L_y$ are pairwise disjoint on $X_n$. Similarly, the ones for $M_t, M_x, M_y$ are also pairwise disjoint on $X_n$.

\begin{lem}\label{L:basic} The automorphism $F_n: X_n \to X_n$ satisfies the following properties: 
\begin{enumerate}
\item  $F_n(p) \in a_1$ if and only if $p \in L_y$,
\item  $F_n(p) \in b_1$ if and only if $p \in L_x$,
\item  $F_n(p) \in c_1$ if and only if $p \in L_t$,
\item $F_n(p) \in a_i$ if and only if $p\in a_{i-1}$, for $i=2, \dots, n$, 
\item  $F_n(p) \in M_x$ if and only if $p \in c_1$,
\item  $F_n(p) \in M_y$ if and only if $p \in b_1$, and
\item  $F_n(p) \in M_t$ if and only if $p \in a_n$.
\end{enumerate}
\end{lem}

Thus, for a given curve in $X_n$, its intersections with exceptional curves  $a_1, \dots, a_n, b_1, c_1$ and three lines $L_x,L_y,L_z$ determines the number of intersections of the $F_n$-image with each exceptional curve $a_1, \dots, a_n, b_1, c_1$ and three lines $M_x, M_y, M_z$. Additionally, we use three curves $M_t, M_x, M_y$ as barriers. In Figure \ref{F:exceptional}, the strict transforms of $L_t, L_y, L_x$ are depicted on the left-hand side, and the strict transforms of $M_t, M_x, M_y$ are on the right-hand side. To determine $\hat F_{n*}\gamma$ for each generator $\gamma \in \pi_1(\hat X_n)$, we use the following steps :
\begin{enumerate}
\item Let $\gamma$ be the curve representing the corresponding generator of $\pi_1(\hat X_n)$.
\item Determine the intersection of $\gamma$ with $L_t, L_x, L_y, c_1, b_1$, and  $a_i,i=1, \dots, n$. This can be done by placing the curve $\gamma$ on the left-hand side of Figure \ref{F:exceptional}.
\item Using $M_t, M_x, M_y$ as barriers, determine the image $\hat F_n(\gamma)$ and its intersections with $c_1, b_1$, and  $a_i,i=1, \dots, n$. This can be done by placing the curve $\gamma$ on the right-hand side of Figure \ref{F:exceptional}.
\item Utilize the reading curves to determine the $\pi_1$ class of $\hat F_n(\gamma)$. 
\end{enumerate}

Let us outline the procedure using only $a_1$ and $c_1$ to avoid repetition of essentially identical computations. 

\subsubsection{Image of a generator $a_1$} Let $\tilde a_1$ be the curve representing the generator $a_1$. The curve, $\tilde a_1$, intersects $L_y, a_1, L_x, L_t$ in cyclic order by Lemma \ref{L:basic} : \[ \tilde a_1 :   L_y \ \to \ a_1  \ \to \ L_x  \ \to \ L_t .\] Consequently, $F_n(\tilde a_1)$ intersects $a_1, a_2, b_1$, and $c_1$ in cyclic order. Since $F_n(\tilde a_1)$ can not intersect $M_t, M_x, M_y$ or the invariant cubic, it must cross $a_1$ and $b_1$ from top to bottom and cross $a_2$, and $c_1$ from bottom to top, as depicted in Figure \ref{F:a1}. In the figure, curves $\tilde a_1$ and $F_n(\tilde a_1)$ are shown as thick dash-dotted directed curves.  Therefore we have: \[ \hat F_{n*} a_1 = [ F_n \tilde a_1]  = a_1 a_2^{-1} b_1 c_1^{-1}. \]

\begin{figure}
\includegraphics[width=6in]{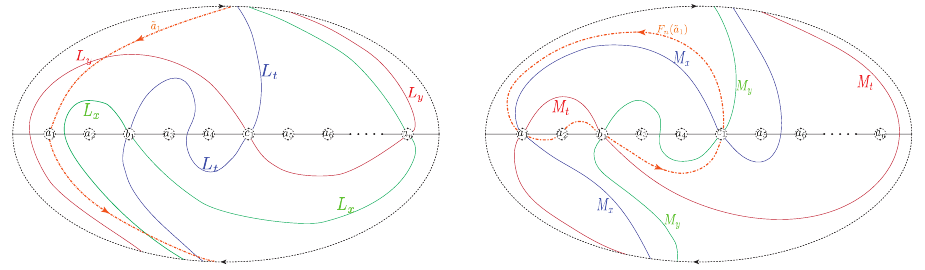}
\caption{LHS: $\tilde a_1$ with forward exceptional curves $L_t, L_x, L_y$; RHS:  $F_n(\tilde a_1)$ with backward exceptional curves $M_t, M_x, M_y$ }
\label{F:a1}
\end{figure}

\subsubsection{Image of a generator $c_1$} Similarly, let $\tilde c_1$ be the curve representing the generator $c_1$. The curve, $\tilde c_1$, only intersects $ c_1$ and  $L_x$ in cyclic order by Lemma \ref{L:basic} : \[ \tilde c_1 :   \ c_1  \ \to \ L_x  .\] Thus, its image, $F_n(\tilde c_1)$, only intersects $M_x$ and $b_1$ in cyclic order. As $F_n(\tilde c_1)$ can not intersect $M_t, M_y$, and the invariant cubic, it must cross $b_1$  from top to bottom as shown in Figure \ref{F:c1}. Therefore, we have: \[ \hat F_{n*} c_1 = [ F_n \tilde c_1]  = b_1. \]

\begin{figure}
\includegraphics[width=6in]{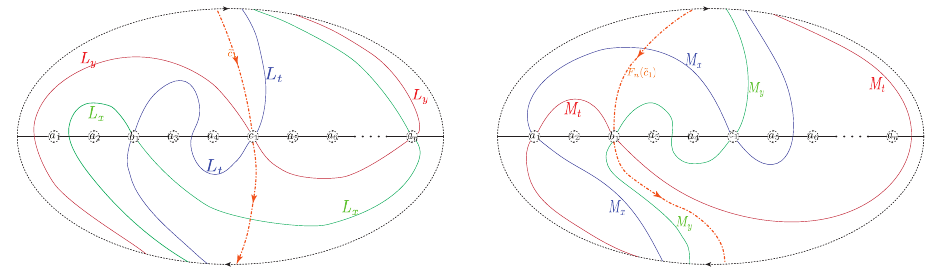} 
\caption{LHS: $\tilde c_1$ with forward exceptional curves $L_t, L_x, L_y$, RHS:  $F_n(\tilde c_1)$ with backward exceptional curves $M_t, M_x, M_y$ }
\label{F:c1}
\end{figure}

\subsection{Action $\hat F_{n*}$ on $\pi_1(\hat X_n)$} By applying the same procedure to all generators, we derive the induced action on the fundamental group of $\hat X_n$:

\begin{prop}\label{P:faction}The induced action $\hat F_{n*}$ on $\pi_1(\hat X_n)$ is defined as follows: 
\begin{equation}\label{E:factionP} \hat F_{n*}\ \ :\ \  \left\{\ \  \begin{aligned} & a_1 \mapsto a_1 a_2^{-1} b_1 c_1^{-1} \\
& a_2 \mapsto a_1 b_1^{-1} a_3 c_1^{-1}\\
& a_3 \mapsto a_1 c_1 a_4^{-1}b_1 \\
& a_4 \mapsto a_1 a_5 c_1^{-1} b_1\\
& a_i \mapsto c_1 a_{i+1}^{-1}a_1^{-1} b_1\quad i=5, 6, \dots, n-1\\
&a_n \mapsto c_1\\
&b_1 \mapsto a_1\\
&c_1 \mapsto b_1.\\
\end{aligned}\right.\end{equation}
\end{prop}

The growth rate $\rho(h_*)$ of the action $h_*$ on the fundamental group $\pi_1(S)$ of a surface $S$ is defined \cite{Bowen} by 
 \[ \rho (h_*) := \sup_{g \in G} \{ \limsup_{n \to \infty} ( \ell_G (h^n_* g))^{1/n}\} \]
 where $G$ is a set of generators of $\pi_1(S)$, and $\ell_G(w)$ is the minimal length among all words representing $w$ with respect to a fixed set of generators of $G$. 

\begin{prop}\label{P:fgrow} The growth rate, $ \rho(F_{n*})$, of $\hat F_{n*}$ on  $\pi_1(\hat X_n)$ is given by $\lambda_n$, the largest real root of $ t^{n} (t^3-t-1) + t^3+t^2-1$.
\end{prop}

\begin{proof}
By Bowen \cite{Bowen}, the logarithm of the growth rate of the induced action on the fundamental group gives a lower estimate of the topological entropy. By Theorem \ref{T:rr}, the topological entropy $h_{top}(F_n)$ is $\log \lambda_n$. Since $\hat F_n$ is the restriction of $F_n$ on $\hat X_n \subset X_n$, we have \[ \log \rho(F_{n*}) \le h_{top} (F_n) = \log \lambda_n.  \] 
On the other hand, the induced action on the abelianization of the fundamental group is given by $(n+2) \times (n+2)$ matrix $M_n$ using the ordered basis $\langle a_1, a_2, \dots, a_n, b_1, c_1 \rangle$:
\[M_n \ = \  \begin{bmatrix} 1&1&1&1&-1& \cdots  &-1&0&1 &0\\ -1 &0&&&&&&&&0 \\&1 &&&&&&&&\\ &&-1&&&&&0&&\\&&&1&&&&&&\\&&&&-1&&&&&\\&&&&&\ddots&&&&\\&&&0&&&-1&&&\\1&-1&1&1&1&\cdots&1&0&0&1\\-1&-1&1&-1&1&\cdots&1&1&0&0
\end{bmatrix}. \]
Direct computation shows that the characteristic polynomial of $M_n$ is given by $( (-t)^{n} (-t^3+t-1) - t^3+t^2-1)/(t+1)$. Thus, the spectral radius of $M_n$ is the largest root $\lambda_n$ of $ t^{n} (t^3-t-1) + t^3+t^2-1$. The growth rate of the action on the first homology group is bounded above by the growth rate of the action on the fundamental group.  It follows that 
\[ \log \lambda_n \le \log \rho(F_{n*}) \le \log \lambda_n.\]
\end{proof}

\medskip
Bestvina and Handel \cite{Bestvina-Handel:1992} proved that if an automorphism of a free group is irreducible, then its fixed subgroup is either trivial or infinite cyclic (in particular, its rank is at most~$1$). Their proof proceeds by showing that an irreducible topological representative can be turned, by a sequence of folds, into an irreducible train-track map, and that such a train-track has at most one closed Nielsen path.

\begin{thm}[Bestvina-Handel {\cite{Bestvina-Handel:1992}}] \label{T:BH}
If $\phi \colon F_n \to F_n$ is an irreducible automorphism of the free group $F_n$, then the fixed subgroup
\[
\mathrm{Fix}(\phi) = \{x \in F_n \mid \phi(x) = x\}
\]
is either trivial or infinite cyclic. Equivalently, $\operatorname{rank}(\mathrm{Fix}(\phi)) \le 1$.
\end{thm}

For each $n\ge 8$, define an element $\mathcal{N}_n \in \pi_1(\hat X_n)$ by
\[
\mathcal{N}_n =
\begin{cases}
a_1 a_2^{-1} b_1 a_3^{-1} a_4 c_1^{-1} a_5 a_6^{-1} \cdots a_{n-1}^{-1} a_n, 
& \text{if $n$ is odd}, \\[6pt]
a_1 a_2^{-1} b_1 a_3^{-1} a_4 c_1^{-1} a_5 a_6^{-1} \cdots a_{n}^{-1} 
a_1^{-1} a_2 b_1^{-1} a_3 a_4^{-1} c_1 a_5^{-1} a_6 \cdots a_{n-1}^{-1} a_n, 
& \text{if $n$ is even}.
\end{cases}
\]

\begin{prop}\label{P:irreducible}
For each $n \ge 8$, if $x \in \pi_1(\hat X_n)$ is fixed by $\hat F_{n*}$, then 
\[
x = \mathcal{N}_n^k \quad \text{for some } k \in \mathbb{Z}.
\]
\end{prop}


\begin{proof}
From the induced action in \eqref{E:factionP} we have the chain
\(a_n \to c_1 \to b_1 \to a_1\).
Hence each of \(b_1, c_1\), and \(a_i\) for \(2\le i\le n\) reaches \(a_1\) in the directed
transition graph of \(\hat F_{n*}\).
Conversely, starting from \(a_1\) one reaches every generator after finitely many iterates, so every vertex reaches every other vertex.
Therefore the transition digraph is strongly connected, and the transition matrix of
\(\hat F_{n*}\) is irreducible.

By Bestvina-Handel's theorem 
the fixed subgroup of \(\hat F_{n*}\) has rank at most \(1\).
A direct computation using \eqref{E:factionP} shows that
\(\hat F_{n*}(\mathcal N_n)=\mathcal N_n\),
so the fixed subgroup is infinite cyclic, generated by the primitive root of \(\mathcal N_n\)
(in particular, by \(\mathcal N_n\) itself if it is not a proper power).
\end{proof}

%% file: dehn.tex
It is a well-known fact that the mapping class group $\mathcal{M}_{g, n}$ of an orientable surface of genus $g$ with $n$ punctures is finitely generated by Dehn twists \cite{Dehn:1938, Lickorish:1964}. For detailed discussions and proofs, we refer a well-written book by Farb and Margalit \cite[Chapter~4]{Farb-Margalit:2012}. In this section, we construct a homeomorphism on $\hat X_n$ using a set of two-sided curves. 

\subsection{Dehn twists}\label{SS:twist} Let $A= S^1 \times [0,1]$ be the annulus. With $(\theta, t) \mapsto (\theta, t+1)$, we embed $A$ into the plane with polar coordinates $(\theta, r)$, and the orientation of $A$ is given by the standard orientation of the plane. The twist map $\tau:A \to A$ is defined by $\tau(\theta, t) = (\theta + 2 \pi t, t)$. 

Suppose $\gamma$ is a two-sided simple closed curve on a surface $X$, and let $\phi: A \to N_\gamma$ be an orientation-preserving homeomorphism between $A$ and a regular neighborhood $N_\gamma$ of $\gamma$. The Dehn twist $\tau_\gamma$ about $\gamma$ is a homeomorphism on $X$ defined by the following formula:
\[ \tau_\gamma(x) \ = \ \left \{ \ \ \begin{aligned} &\phi \circ \tau\circ \phi^{-1}(x) \qquad &\text{ if  } x \in \phi(A) = N_\gamma \\ & x  &\text{if  } x \in X \setminus N_\gamma. \phantom{AA,}\\ \end{aligned} \right. \]

\subsection{A set of two-sided curves}\label{SS:gammas} A curve passing through cross-caps an even number of times is two-sided (i.e., a core of annulus), while a curve passing through cross-caps an odd number of times is one-sided (i.e., a core of M\"{o}bius band). Consider the curves $\gamma_i, i=1, \dots, n-2$ as shown in Figure \ref{F:curves}. The curve $\gamma_1$ is a blue dash-dotted curve passing through two cross-caps $b_1$ and $c_1$. The curve $\gamma_2$ is a purple dash-dotted curve passing through two cross-caps $c_1$ and $a_n$. $\gamma_3$ is a red dash curve passing through four cross-caps $b_1, c_1, a_n$ and the boundary oval curve. $\gamma_{n-1}$ is a solid orange curve passing through cross-caps four times in the order $a_4, c_1, a_5$ and $c_1$ again. Similarly, $\gamma_{n+1}$ passes through $a_2, b_1, a_3$ and $b_1$ in that order. All other curves passes through two cross-caps, $a_i$ and $a_{i+1}$ for $i=1,3,5,6, \dots, n-1$. Figure \ref{F:curves} is drawn with the invariant curve; however, all $\gamma_i$ curves are disjoint from the invariant cubic and thus belong to $\hat X_n$ for $i=1, \dots, n+2$. The incidence graph of these curves is shown in Figure \ref{F:incident}. As you can see, the incidence graph is identical to the $E_{n+2}$ diagram. 

\begin{figure}
\includegraphics[width=5in]{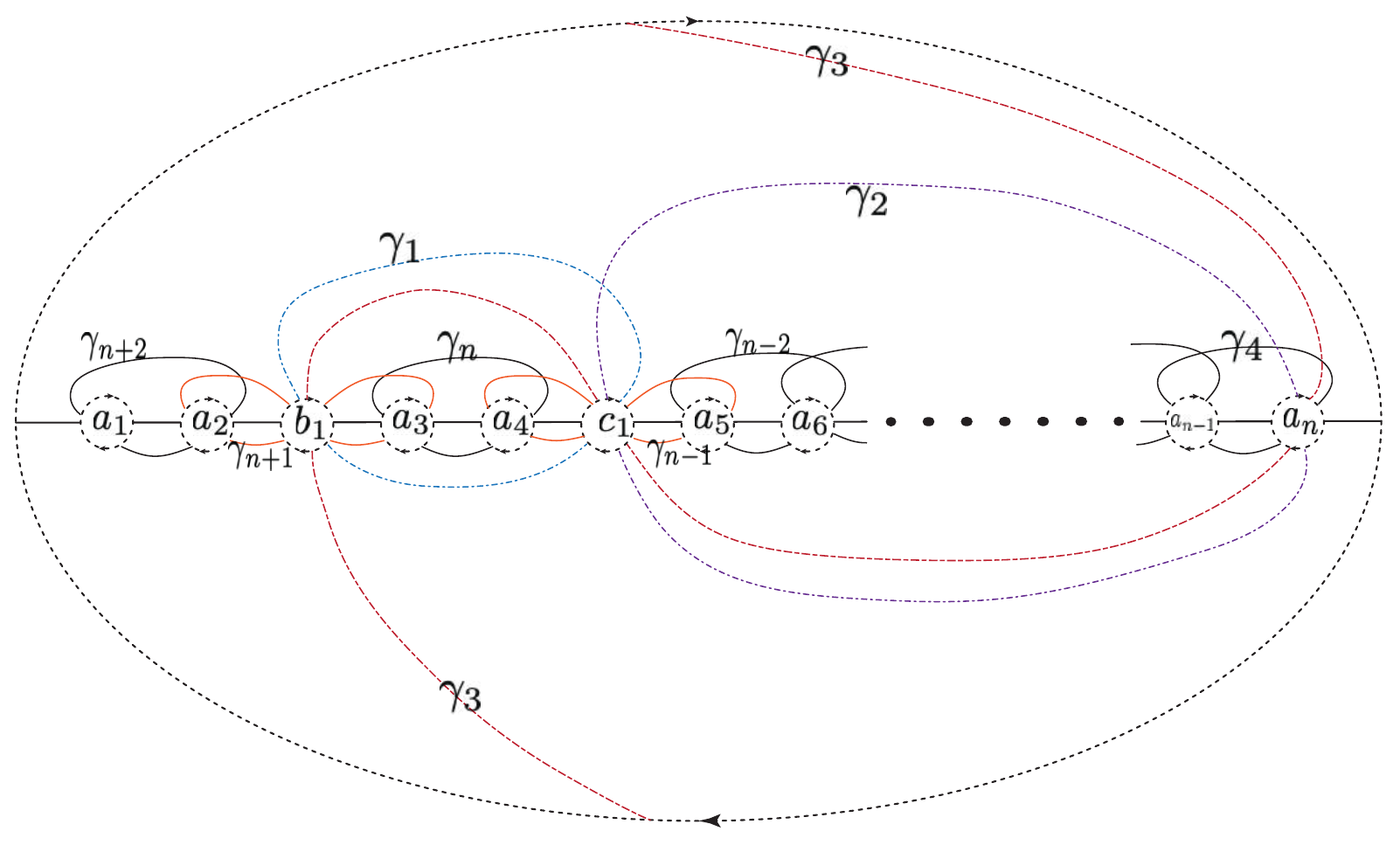}
\caption{The incidence graph of the set of curves $\gamma_i, i=1, \dots, n-2$  on $\hat X_n$}
\label{F:curves}
\end{figure}

\begin{figure}
\includegraphics[width=4in]{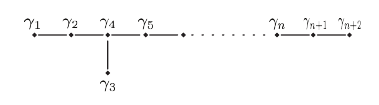}
\caption{The incident graph of the set of curves $\gamma_i, i=1, \dots, n-2$  on $\hat X_n$}
\label{F:incident}
\end{figure}

Let $\tau_i$ be the Dehn twist about $\gamma_i$, $i=1, \dots, n+2$, and let $T_n$ is a composition of $n+2$ Dehn twists, $\tau_i$, $i=1, \dots, n+2$:
\begin{equation}\label{E:twist}
T_n\ =\  \tau_{n+2} \circ \tau_{n+1} \circ \tau_n \circ \cdots \circ \tau_4\circ \tau_3\circ \tau_2 \circ \tau_1.\end{equation}

\subsection{Action on the fundamental group} With the help of reading curves, we can compute the induced action $\tau_{i*}$ on $\pi_1(\hat X_n)$. The procedure is essentially identical to the cases of rational surface automorphisms in Section \ref{S:rat}, let us list the induced actions. 

\[\tau_{1*}\  :\  \left\{ \begin{aligned} &b_1 \mapsto b_1 c_1^{-1}b_1,\\& c_1 \mapsto b_1, \\& a_i \mapsto b_1 c_1^{-1} a_i c_1^{-1} b_1, \quad i=3,4,\\&\text{fix   otherwise,} \end{aligned}\right. \quad\text{and} \quad \tau_{2*}\  :\  \left\{ \begin{aligned} &c_1 \mapsto c_1 a_n^{-1}c_1,\\& a_n \mapsto c_1, \\& a_i \mapsto c_1 a_n^{-1} a_i a_n^{-1} c_1, \quad i=5,6, \dots, n-1,\\&\text{fix   otherwise.} \end{aligned}\right. \]
\[\tau_{3*}\  :\  \left\{ \begin{aligned} &a_i \mapsto a_i b_1^{-1} c_1 a_n^{-1} \quad i =1,2,\\& a_i \mapsto b_1 a_n c_1^{-1} a_i \quad i= 3,4, \\& a_i \mapsto a_i a_n^{-1} b_1^{-1} c_1 \quad i=5,6, \dots, n-1,\\&\text{fix   otherwise.} \end{aligned}\right.\]
For $i= 4,5,\dots, n-2, n, n+2$, 
\[ \tau_{i*}\ :\  \left\{ \begin{aligned} &a_{n+4-i} \mapsto a_{n+3-i},\\& a_{n+3-i} \mapsto a_{n+3-i} a_{n+4-i}^{-1} a_{n+3-i},\\&\text{fix   otherwise.} \end{aligned}\right.\]
Lastly, we have:
\[ \tau_{n-1*} \ :\  \left\{ \begin{aligned} &a_{4} \mapsto a_4 c_1^{-1} a_5 c_1^{-1} a_4,\\& a_5 \mapsto c_1 a_4^{-1} c_1,\\&\text{fix   otherwise,} \end{aligned}\right.\quad\text{and} \quad  \tau_{n+1*} \ :\  \left\{ \begin{aligned} &a_{2} \mapsto a_2 b_1^{-1} a_3 b_1^{-1} a_2,\\& a_3 \mapsto b_1 a_2^{-1} b_1,\\&\text{fix   otherwise.} \end{aligned}\right.\]

Combining the above computations, we have:
\begin{prop}\label{P:taction} The induced action $T_{n*}$ on $\pi_1(\hat X_n)$ is defined by the following formula: 
\[ T_{n*}\ \ :\ \ \left\{\ \  \begin{aligned}  & a_1 \mapsto a_1 a_2^{-1} b_1 c_1^{-1} \\
& a_2 \mapsto a_1 b_1^{-1} a_3 c_1^{-1}\\
& a_3 \mapsto a_1 c_1 a_4^{-1}b_1 \\
& a_4 \mapsto a_1 a_5 c_1^{-1} b_1\\
& a_i \mapsto c_1 a_{i+1}^{-1}a_1^{-1} b_1\quad i=5, 6, \dots, n-1\\
&a_n \mapsto c_1\\
&b_1 \mapsto a_1\\
&c_1 \mapsto b_1.\\
\end{aligned}\right.\]
\end{prop}

\subsection{Penner's construction}\label{pen} Penner's construction is a well-known method for constructing pseudo-Anosov mappings. (See \cite{Penner:1988,Farb-Margalit:2012,Liechi-Strenner:2020}.) A \textit{multicurve} is a finite collection of disjoint non-isotopic essential simple closed curves on a surface. 

We say two simple closed curves $\alpha, \beta$ on a surface are in \textit{minimal position} if they have the minimal intersection number in their homotopy classes: \[ | \alpha \cap \beta |   = \ \min_{\alpha' \sim \alpha,\,\, \beta' \sim \beta } |\alpha' \cap \beta'|. \] 
The minimum on the right-hand side is called the geometric intersection number $\iota(\alpha, \beta)$ of two curves, $\alpha, \beta$. Also, we say a collection of simple closed curves $C$ on a surface $S$ is \textit{filling} if the curves are in pairwise minimal position and each component of $S\setminus C$ is either a disk or a once-punctured disk, or equivalently if every essential curve intersects $S$
\begin{thm}\cite{Penner:1988}\label{T:pen} Let $A=\{\alpha_i, i=1, \dots, n\}$ and $B=\{\beta_j, j=1, \dots, m\}$ be a pair of filling multicurves on an orientable surface $S$. Then, any product of right (positive) Dehn twists about $\alpha_i$ and left (negative) Dehn twists about $\beta_j$ is pseudo-Anosov if all $n+m$ twists appear at least once. 
\end{thm} 

\begin{figure}
\includegraphics[width=6in]{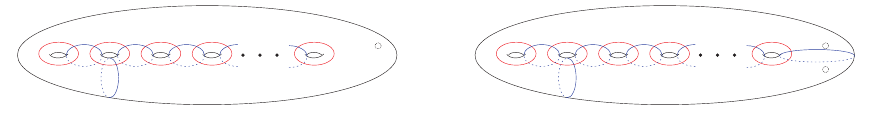} 
\caption{Curves in $A$ and $B$. LHS is shown the case when $n$ is even and RHS is for the case when $n$ is odd. }
\label{F:ab}
\end{figure}

Because of the simplicity of the incident graph, we can choose two multi-curves: 
\[ A\ =\  \{\gamma_1\} \cup \{ \gamma_{i}: i \text{ is even, and } 3 \le i \le n+2\}, \quad  B\ =\ \{ \gamma_2\} \cup \{ \gamma_{i}: i \text{ is odd, and } 3 \le i \le n+2\}. \]
Since $\hat X_n$ is homeomorphic to an orientable surface of genus $g= \lfloor (n+2)/2 \rfloor$ with one or two punctures, we can arrange $A$ and $B$ as shown in Figure \ref{F:ab} . In Figure \ref{F:ab}, the dotted circles are punctures, the curves in $A$ are drawn in red curves, and the curves in $B$ are shown in blue curves.
Define a Dehn twist $\tilde \tau_i $ by \[ \tilde \tau_i \ = \ \left\{ \begin{aligned} &\tau_{\gamma_i} \qquad\,\, \text{if } \gamma_i \in A\\&\tau^{-1}_{\gamma_i} \qquad \text{if } \gamma_i \in B. \end{aligned} \right. \]

Suppose $n\ge 8$ and suppose $\sigma \in \Sigma_{n+2}$ is a permutation. Let $T_\sigma$ be a product of Dehn twists defined by 
\[ T_\sigma \ = \ \tilde \tau_{\sigma(1)}\circ \tilde \tau_{\sigma(2)} \circ \cdots \circ \tilde \tau_{\sigma(n+2)}.\]
Then, by Penner's construction in Theorem \ref{T:pen}, it is clear that the product $T_{\sigma}$ is pseudo-Anosov. 


For Penner's construction, there is a simple and elegant method to compute the stretching factor using an incident matrix of curves in $A \cup B$. (See \cite{Penner:1988,  Strenner:2017}.) A brief summary is as follows. Let $\Omega$ be the intersection matrix of $\gamma_i$'s; the $i,j$ entry of $\Omega$ is the geometric intersection number of $\gamma_i$ and $\gamma_j$. Let $D_i$, $i=1, \dots, n+2$ be an $(n+2) \times (n+2)$ matrix such that the $i,i$ entry equals $1$ and all other entries equal zero. Also, let $Q_i, 1\le i \le n+3$ be an $(n+2)\times (n+2)$ matrix defined by \[ Q_i = I_{n+3} + D_i \Omega,\] where $I_{n+3}$ is an $(n+2)\times (n+2)$ identity matrix. The stretch factor of $T_\sigma$ is the spectral radius of the product $Q_{\sigma, n} =Q_{\sigma(1)} Q_{\sigma(2)} \cdots Q_{\sigma(n+2)}$.

For example, suppose $\sigma=Id$. Then for each $n\ge 8$,
\[ Q_{Id,n} \ =\ \begin{bmatrix} 2&2&1&2&2&\cdots&&2&1\\1&2&1&2&2&\cdots&&2&1\\0&1&2&2&2&\cdots&&2&1\\0&1&1&2&2&\cdots&&2&1\\0&0&0&1&2&\cdots&&2&1\\ & & & & \ddots& && \vdots&\vdots\\& & & & & \ddots&& \vdots&\vdots\\ &&&&&&1&2&1\\&&&&&&&1&1\end{bmatrix}. \]
Since this family of matrices exhibits a consistent pattern in its entries, we can obtain a formula for the characteristic polynomial. Let us only state the answer since this has been done using the cofactor expansion. The characteristic polynomial $ \chi_n(t)$ of $Q_{Id, n}$ is given by 
 \begin{equation}\label{E:charp} \chi_n(t) = t b_{n+3}+(t^2-3 t+2) b_{n+2} + (2t-1) b_{n+1},\end{equation}
 where \[b_n=-(t-1)b_{n-1}-t b_{n-2}, \qquad b_1=1-t,\qquad \text{and } \qquad b_2=1-3 t+t^2.\] The sequence $\{ b_n\}$ of polynomials is a sequence of determinants of matrices that repeatedly appeared when computing the determinant of $Q_{Id, n} - t Id$.  Suppose $B_n$ is an $n \times n$ matrix defined by 
\[ B_n\ = \begin{bmatrix} 2-t & 2&2&2&\cdots&2&1\\ 1&2-t&2&2&\cdots&2&1\\0&1&2-t&2&\cdots&2&1\\ 0&0& 1&2-t&\cdots&2&1\\ \vdots &&&\ddots& &\vdots&\vdots \\&&&&1&2-t&1\\ &&&&&1&1-t \end{bmatrix}, \]
then $b_n =\det(B_n)$ is the determinant of $B_n$.
It follows that the stretch factor of $T_{Id}$ is the largest real root of $\chi_n(t)$.

\begin{rem} Based on computer experiments, we noticed that for $n \le 15$, for all permutations $\sigma \in \Sigma_{n+2}$, the stretch factor of $T_{\sigma}$ equals to the stretch factor of $T_{Id}$.

\vspace{1ex}
\begin{center}
\begin{tabular}{|c|c|||c|c|}
\hline
$n$& the stretch factor of $T_\sigma$, $\sigma \in \Sigma_{n+2}$&$n$& the stretch factor of $T_{Id}$\\
\hline
$8$ & $5.70407$&$16$ & $5.79333$\\
\hline
$9$&$5.72752$&$17$ & $5.79712$\\
\hline
$10$&$5.74492$&$18$ & $5.80032$\\
\hline
$11$&$5.75853$&$19$ & $5.80305$\\
\hline
$12$&$5.76853$&$20$ & $5.80541$\\
\hline
$13$&$5.77657$&$21$ & $5.80745$\\
\hline
$14$&$5.78339$&$22$ & $5.80923$\\
\hline
$15$&$5.78882$&$23$ & $5.8108$\\
\hline
\end{tabular}
\end{center}

\vspace{1ex}
With the polynomials $\chi_n(t)$ defined in (\ref{E:charp}), finding the stretch factor of $T_{Id}$ is straightforward to implement in a computer program.
While we do not yet have conclusive proof, computer experiments suggest that the stretch factor does not depend on the permutation $\sigma \in \Sigma_{n+2}$, and the sequence of the stretch factor of $T_{Id}$ is monotonly increasing in $n$. 
\end{rem}

%% file: thma.tex
Let $S=S_g$ be a closed surface of genus $g\ge 1$ without boundary, and let $\text{Mod}(S)$ be the mapping class group, the group of isotopy classes of orientation-preserving homeomorphisms of $S$. Also, let $\text{Mod}^\pm(S)$ be the group of isotopy classes of all (orientation-preserving and orientation-reversing) homeomorphism. Then, $\text{Mod}(S)$ is an index $2$ subgroup of $\text{Mod}^\pm(S)$. By the Dehn-Nielsen-Bare theorem, there is an isomorphism between $\text{Mod}^\pm(S)$ and the outer automorphism group, $\text{Out}(\pi_1(S))$. 

The surface $\hat X_n$ is an orientable surface with one or two punctures. There is a version of the Dehn-Nielson-Bare theorem for the punctured surfaces. Suppose $\text{Out}^\star(\pi_1(S))$ is the subgroup of the outer automorphism group of $\pi_1(S)$ consisting of elements that preserve the set of conjugacy classes of the simple closed curves surrounding individual punctures. 

\begin{thm}\cite[Theorem~8.8]{Farb-Margalit:2012} Let $S=S_{g,n}$ be a hyperbolic surface of genus $g$ with $n$ punctures. Then the natural map \[ \text{Mod}^\pm(S) \ \to \text{Out}^\star(\pi_1(S))\] is an isomorphism. 
\end{thm}

Thus, if $f$ and $g$ are two orientation-preserving homeomorphisms on $S$ with the identical induced action on $\pi(S)$, we can conclude that $f$ and $g$ are isotopic to each other. It follows that we have:

\begin{thm}\label{T:iso}
Suppose $n\ge 8$. Then two homeomorphisms $\hat F_n$ (defined in Section \ref{S:rat}) and $T_n$ (defined in (\ref{E:twist}) in Section \ref{S:twists}) on $\hat X_n$ are isotopic. 
\end{thm}

\begin{proof}
First recall Theorem \ref{T:rat}, which states that there is a meromorphic two form $\omega$ on $X_n$ such that (1) $\omega$ has a simple pole along the invariant cubic and no other poles or zeros, and (2) $F_n$ satisfies  $F_n^* \omega = \lambda_n \omega$ with $\lambda_n>1$. It follows that $\hat F_n$ on $\hat X_n$ is an orientation preserving diffeomorphism. For both $F_n$ and $T_n$, the induced actions on $\pi_1(\hat X_n)$ were computed in Proposition \ref{P:faction} and Proposition \ref{P:taction}. Comparing Proposition \ref{P:faction} and Proposition \ref{P:taction}, we see that their induced actions are identical. 
\end{proof}

Let us recall the remarkable Nielsen-Thurston classification on an orientable surface with punctures. Suppose $S_{g,n}$ is an orientable surface of genus $g$ with $n$ punctures. 

\begin{thm}\cite[Theorem~13.2]{Farb-Margalit:2012} Let $g, n\ge 0$. For mapping class $f \in \text{Mod}(S_{g,n})$ is periodic, reducible, or pseudo-Anosov. 
\end{thm}

\begin{prop}\label{P:pseudo}
Let $g=\lfloor (n+2)/2 \rfloor$.  
There exists a punctured surface
\[
  \Sigma_n \;\cong\; S_g\smallsetminus\{\text{at least one point}\}
\]
such that $\widehat F_n$ restricts to a pseudo-Anosov homeomorphism on $\Sigma_n$ whose stretch factor is
$\lambda_n$, the largest real root of the polynomial
$p_n(t)=t^{n}(t^3-t-1)+t^3+t^2-1$.
\end{prop}

\begin{proof}
Because $F_n$ possesses a unique invariant simple closed curve
$C_n\subset X_n$ and
$\widehat X_n=X_n\setminus C_n$, the homeomorphism
$\widehat F_n$ has no essential invariant curves.  In particular, by Proposition \ref{P:irreducible} we see that every
essential simple closed curve in $\widehat X_n$ has
infinite order.

It follows that $\hat F_n$ is not reducible. If the induced action $\hat F_{n*}$ on $\pi_1(\hat X_n)$ is periodic, then so is the induced action of $F_{n*}$ on $X_n$. However, this is not true due to Theorem \ref{T:rat}. By Nielsen-Thurston classification, we conclude that $\hat F_n$ on $\hat X_n$ is pseudo-Anosov. 

By Theorem \ref{T:rr}, the growth rate of the induced action of $\hat F_n$ on homology is given by $\lambda_n$. Hence, for any $g$ in the mapping class of $\hat F_n$, the topological entropy satisfies \[ h_{top}(g) \ge \log \lambda_n.\]
Let $\phi$ be a pseudo-Anosov diffeomorphism in the mapping class of $\hat F_n$. Then $\phi$ realizes the minimal entropy in the class, which is equal to $\log \delta$, where $\delta$ is the stretch factor of $\phi$.

By Proposition \ref{P:fgrow}, we know that $F_n$ realizes this minimal entropy, so $\log \delta = \log \lambda_n$. 
Hence $\delta = \lambda_n$ and we conclude that the stretch factor of $\hat F_n$ is equal to $\lambda_n$.
\end{proof}

%
%

We say a pseudo-Anosov map is \textit{coronal} if its stretch factor has a Galois conjugate on the unit circle. Shin and Strenner \cite{Shin-Strenner:2015} shows that there is an obstacle for a pseudo-Anosov map to arise from Penner's construction. 

\begin{prop}\cite[Corollary~2.5]{Shin-Strenner:2015} A coronal pseudo-Anosov mapping class has no power of them coming from Penner's construction.
\end{prop}

\begin{thm}\label{T:nopen} The homeomorphism $\widetilde F_n$ does not arise from Penner's construction. 
\end{thm}

\begin{proof}
The stretch factor of $\widetilde F_n$ is a Salem number for all $n\ge 8$. It follows that the mapping class of $\widetilde F_n$ is coronal. By Shin and Strenner \cite{Shin-Strenner:2015}, we conclude that $\widetilde F_n$ does not arise from Penner's construction.
\end{proof}

\begin{proof}[Proof of the Main Theorem]
By construction, \(F_n\) preserves the unique invariant simple closed curve \(C_n\); hence the mapping class of \(F_n\) is reducible.  
Proposition~\ref{P:cut} shows that the cut surface
\[
  \widehat X_n \;=\; X_n\setminus C_n
\]
is an orientable surface of genus \(g=\lfloor(n+2)/2\rfloor\) with  
\emph{one} puncture when \(n\) is even and \emph{two} punctures when \(n\) is odd.

Combining Theorem~\ref{T:iso} with Proposition~\ref{P:pseudo}, we obtain a surface  
\[
  \Sigma_n \;\cong\; S_g \setminus\{\text{at least one point}\}
\]
on which \(\widehat F_n\) descends to a pseudo-Anosov homeomorphism
\[
  \widetilde F_n\colon \Sigma_n \longrightarrow \Sigma_n
\]
whose stretch factor is \(\lambda_n\).  Moreover, \(\widehat F_n\) (and hence \(\widetilde F_n\)) is isotopic to
\[
  T_n \;=\; \tau_{\gamma_1}\circ\cdots\circ\tau_{\gamma_{n+2}},
\]
the composition of the \(n+2\) positive Dehn twists defined in~\eqref{E:twist}.

Finally, Theorem~\ref{T:nopen} shows that the mapping class of \(\widehat F_n\) does \emph{not} arise from Penner's construction.
\end{proof}

\begin{rem} The surface $\hat X_n$ is an orientable surface $S_g$ of genus $g=\lfloor (n+2)/2 \rfloor$ with one or two punctures. Thus, one can consider $T_n$ as a homeomorphism on a closed surface $S_g$ with one or two marked points. Since a pseudo-Anosov map on a surface with Euler characteristic $<0$ must have a singularity, the homeomorphism $T_n$ must have singularities at the marked points. By blowing up those singularities, we recover the mapping on a non-orientable rational surface. 
\end{rem}

Recall a construction of Thurston \cite{Thurston:1988, Long:1985} that was given by a product of two positive multi-twists. 
Suppose $A$ and $B$ are two multi-curves on a connected finite-type oriented surface $S$. Let $T_A$ be the product of positive (i.e., right) Dehn twist about the components of $A$. The subgroup $\langle T_A, T_B\rangle < \text{Mod}(S)$ generated by $T_A$ and $T_B$ were studied by Leininger \cite{Leininger:2004}.

\begin{thm}\cite[Theorem~6.2]{Leininger:2004}
For any surface $S$, any multi-curves $A,B$, and any pseudo-Anosov element \[ \phi \in \langle T_A, T_B\rangle < \text{Mod}(S)\] we have the stretch factor $\lambda(\phi) \ge \lambda_8 \approx 1.17628$. Moreover, $\lambda(\phi) = \lambda_8$ precisely when $S$ has genus $5$ (with at most one market point), $\{A, B \} = \{A_L, B_L\}$ where the incident graph of $A_L$, and $B_L$ is given by $E_{10}$ diagram as in Figure \ref{F:incident}, and $\phi$ is conjugate to $(T_A T_B)^\pm$.
\end{thm}

\begin{rem}
For $n=8$, the incident graph of essential simple closed curves $\gamma$s is given by $E_{10}$ diagram as in Figure \ref{F:incident}.  Leininger also showed that most elements in $\langle T_A, T_B\rangle$ are pseudo-Anosov  \cite[Proosition~6.4]{Leininger:2004}.  It is interesting to know if $T_n$ defined in (\ref{E:twist}) is an element of $ \langle T_A, T_B\rangle$ with $A=\{\alpha_i, i=1, \dots, n\}$ and $B=\{\beta_j, j=1, \dots, m\}$. A positive answer leads us to conclude that $\hat F_n$ is conjugate to $(T_A T_B)^\pm$. Also, in the next Section, we discuss how we can see rational surface automorphisms fixing an invariant cusp cubic with pseudo-Anosov maps in an orientable surface with one or two punctures. It would be exciting to see the connection between real rational surface diffeomorphisms and pseudo-Anosov maps obtained by Thurston's construction. 
\end{rem} 


%

%% file: psa.tex
Let $J:\ppc \dasharrow \ppc$ be the involution defined by 
\[ J:[x_1:x_2:x_3] \mapsto [x_2 x_3:x_1 x_3:x_1 x_2]. \]
This map $J$ is an involution with three exceptional lines $\mathcal{E}(J) =\{ \{x_i=0\}, i=1,2,3\}$ and three points of indeterimacy $\mathcal{I}(J) = \{p_1=[1:0:0],p_2=[0:1:0],p_3=[0:0:1]\}$. \[ J: \{ x_i=0\} \setminus \mathcal{I}(J) \mapsto \{ x_j=0\} \cap \{ x_k=0\}, \qquad\text{where} \quad  \{i,j,k\} = \{1,2,3\}.\]
We say a birational map $f:\ppc \dasharrow \ppc$ is \textit{basic quadratic} if $f = T\circ J \circ S^{-1}$ for some $T, S \in \text{GL}(3, \mathbb{C})$. Thus, up to linear conjugacy, the set $\mathcal{Q}_b$ of basic quadratic maps is given by
\[ \mathcal{Q}_b \ = \ \{ L \circ J: L \in \text{GL}(3, \mathbb{C})\}. \]
We say a birational map $f:\ppc \dasharrow \ppc$ \textit{properly fixes} a curve $C$ if (1) $\overline{f (C \setminus \mathcal{I}(J))  } = C$, and (2) the orbit of exceptional curves lies on $C$. Diller \cite{Diller:2011} constructed all quadratic birational maps properly fixing a cubic curve and identified maps birationally equivalent to automorphisms. 

\vspace{1ex}
Suppose $f=L\circ J \in \mathcal{Q}_b$ is a basic quadratic birational map. Then $f$ has three exceptional lines and three points of indeterminacy: \[ \mathcal{E}(f) = \{ E_i = \{ x_i=0\}, i=1,2,3\}, \quad \mathcal{I}(f)= \{ Lp_i, i=1,2,3\}.\] Let $n_i \in \mathbb{N}_{>0} \cup \{\infty\}$ defined by \[ n_i = \min \{ n\ge 1 : \dim f^{n+1} E_i>0 \}, \] where we set $\min \emptyset = \infty$. Since $E_i$ is an exceptional line, $n_i$ is always strictly greater than $0$. It is known that $f$ is birationally equivalent to an automorphism if and only if $n_i < \infty$ for all $i=1,2,3$. 

If $n_i < \infty$ for all $i=1,2,3$, then the set of points $\{f^{n_i} E_i, i=1,2,3\} = \mathcal{I}(f)$. Thus, there is a permutation $\sigma \in \Sigma_3$ such that $f^{n_i} E_i = p_{\sigma(i)}$. We call the three integers $n_1, n_2, n_3$ together with a permutation $\sigma$ the \textit{orbit data} of $f$. 

\vspace{1ex}
It is known \cite{Bedford-Kim:2004} that the orbit data determines the dynamical degree, $\lambda(f)$, defined by the exponential growth rate of the algebraic degree of $n$ fold composition of $f$: \[ \lambda(f) = \lim_{n \to \infty} (\text{deg} f^n)^{1/n}. \] And if $f$ is birationally equivalent to a rational surface automorphism $F: X \to X$, then the topological entropy of  $F$ is given by $\log \lambda(f)$: \[ h_\text{top}(F) = \log \lambda(f) .\] Changing coordinates if necessary, we only need to consider three kinds of permutations: the identity permutation, the transpose between $i$ and $j$, $\sigma_{ij}$, and the cyclic permutation, $\sigma_c$. Let $\mathcal{O}_\dagger$ be the set of orbit data defined by 
\[ \mathcal{O}_\dagger = \{ (n,n,n, \sigma_c), (1,n,n, \sigma_c), (2,2,n, \sigma_c), (n_i=n_j, \sigma_{i,j}), (n,n,n,Id) : n \in \mathbb{N}_{>0} \}. \]
\begin{thm}\cite{Diller:2011}
Suppose $n_1, n_2, n_3$ are three positive integers and $\sigma \in \Sigma_3$ is a permutation. If $(n_1,n_2,n_3, \sigma) \not\in \mathcal{O}_\dagger$,  then there exists $f \in \mathcal{Q}_b$ properly fixing an irreducible cubic $\mathcal{C}$ with a cuspidal singularity such that the orbit data of $f$ is given by $(n_1,n_2,n_3, \sigma) $. In this case, the cubic curve $\mathcal{C}$ is the unique invariant curve of $f$. Furthermore, for each such orbit data, there exists a unique pair of such birational maps $f_r, f_r^{-1} \in \mathcal{Q}_b$, which preserve a real projective plane $\ppr$. 
\end{thm}

A birational map in the above theorem can be constructed using an explicit formula. The details of such construction can be found in \cite{Diller:2011, Kim:2022}. The computer codes (SageMath and Mathematica) generating such birational maps are available at \texttt{https://www.math.fsu.edu/~kim/publication.html}.

\vspace{1ex}
For such $f_r$ in the above theorem, there exists a real rational surface $X$ obtained by blowing up a finite set of points on the invariant cubic curve $\mathcal{C}\in \ppr$. Thus, as in Section \ref{S:rat}, we obtain a diffeomorphism, $F_r$ on an orientable surface, $X\setminus \mathcal{C}$ with one or two punctures.

Diller and Kim \cite{Diller-Kim} identified six orbit data with cyclic permutations whose actions on the homology group are periodic. Klassen and the author \cite{Kim-Klassen} computed the action on the fundamental group for such real diffeomorphisms with cyclic orbit data, showing that these six cases induce actions on the fundamental group with exponential growth rates. Later, Kim and Park \cite{Kim-Park}, using the invariant semigroup, showed that orbit data of the form $(1,m,n)$ with $1+m+n \ge 10$ and a cyclic permutation correspond to induced actions on the fundamental group with exponential growth rate.

\begin{thm}\label{T:kk}\cite{Kim-Park}
Suppose a basic quadratic birational map $f_r\in \mathcal{Q}_b$ fixes a cusp cubic $\mathcal{C}$ properly with the orbit data $(1,n_2,n_3, \sigma_c) \not\in \mathcal{O}_\dagger$ satisfying $1+n_2+n_3 \ge 10$. Also, suppose $F_r: X \to X$ is the automorphism on a real rational surface $X$, which is birationally equivalent to $f_r$ . Also let $\hat F_r: \hat X \to \hat X$ be the diffeomorphism on a cut surface $\hat X = X \setminus \mathcal{C}$.
Then, the induced action $\hat F_{r*}$ on $\pi_1(\hat X)$ has spectral radius $>1$. Furthermore $\hat F_{r*}$ is irreducible.
\end{thm}

\begin{thm}\label{T:mpsa}
Let  
\(
  F_r \colon X \to X
\) 
be the automorphism appearing in Theorem~\ref{T:kk}, and let  
\(\mathcal{C}\subset X\) be its unique invariant simple closed curve system.  
Then the cut surface
\[
  X\setminus\mathcal{C}
\]
is homeomorphic to an orientable surface of genus
\[
  g \;=\;\Bigl\lfloor\frac{n_1+n_2+n_3}{2}\Bigr\rfloor
\]
with  
\(
  k = 1
\)
puncture when \(n_1+n_2+n_3\) is even and  
\(k = 2\) punctures when it is odd.  

Moreover, the restriction
\[
  F_r\bigl|_{\,X\setminus\mathcal{C}}\colon
  X\setminus\mathcal{C}\;\longrightarrow\;X\setminus\mathcal{C}
\]
descends to a pseudo-Anosov homeomorphism on a punctured orientable surface of genus \(g\).
\end{thm}

%

\begin{proof}
By Proposition \ref{P:cut}, the cut surface $X\setminus \mathcal{C}$ is orientable. If $n_1+n_2+n_3$ is even, then the curve $\mathcal{C}$ is $1$-sided and $n_1+n_2+n_3$ is odd, then the curve $\mathcal{C}$ is $2$-sided. Thus, the first assertion follows. 

The induced action, $F_{r*}$ on $\pi_1(X\setminus \mathcal{C})$ is irreducible and not periodic \cite{Kim-Park}. Since the class of the cubic $\mathcal{C}$ is fixed under $F_{r*}$, using  Bestvina-Handel's Theorem, we see that only fixed elements under $F_{r*}$ are given by the power of the class of $\mathcal{C}$. Thus, with the same reasoning as in Proposition \ref{P:pseudo}, we have the desired conclusion. 
\end{proof}

\begin{thm}\label{T:mnp}
Let  
\(
  F_r\colon X \to X
\)
be the automorphism in Theorem~\ref{T:kk} associated with the orbit data
\(
  (2,n,n,\sigma_c)
\)
for some \(n\ge 4\).
Let \(\mathcal{C}\subset X\) be its invariant curve system.

Then the cut surface
\[
  X\setminus\mathcal{C}
\]
is homeomorphic to an orientable surface of genus \(n+1\) with a single puncture.
Moreover, the restriction
\[
  F_r\bigl|_{\,X\setminus\mathcal{C}}\colon
  X\setminus\mathcal{C}\;\longrightarrow\;X\setminus\mathcal{C}
\]
descends to a pseudo-Anosov homeomorphism on that once-punctured surface of genus \(n+1\);
this mapping class does \emph{not} arise from Penner's construction.
\end{thm}
%

\begin{proof}
In this case, Diller and the author \cite{Diller-Kim} showed that the growth rate of the induced action on the homology group is equal to the dynamical degree of the corresponding birational map on $\ppc$. It follows that a Salem number $>1$ gives the growth rate of the induced action on the fundamental group. It follows that $F_r|X \setminus \mathcal{C}$ is coronal. By Shin and Strenner \cite{Shin-Strenner:2015}, we conclude that the induced pseudo-Anosov map does not arise from Penner's construction. 
\end{proof}

\begin{rem} The induced actions on the fundamental groups of the diffeomorphisms in Theorem \ref{T:mpsa} were computed in \cite{Kim-Park}. It would be interesting to know the mapping classes to which the restriction map $F_r|X \setminus \mathcal{C}$ belongs. Also, we only have the lower bound of the growth rates of those induced actions on the fundamental group. Except for some cases, such as mappings in Theorem \ref{T:mnp}, we can not conclude the restriction map $F_r|X \setminus \mathcal{C}$ does not arise from Penner's construction using Shin and Strenner's argument \cite{Shin-Strenner:2015}. 
\end{rem}